\documentclass{article}

\usepackage{lineno}
\usepackage{amsmath}
\usepackage{amsfonts}
\usepackage{amssymb}

\usepackage{color}
\usepackage{todonotes}
\usepackage{epstopdf}% To incorporate .eps illustrations using PDFLaTeX, etc.
\usepackage{subfigure}% Support for small, `sub' figures and tables
\usepackage{dsfont,bm}

\usepackage{makeidx}
\usepackage{graphicx}
\usepackage{lmodern}
\usepackage{amsthm}
\usepackage{amssymb}
\usepackage{setspace}
\theoremstyle{plain}% Theorem-like structures
\newtheorem{theorem}{Theorem}[section]
\newtheorem{corollary}[theorem]{Corollary}
\newtheorem{lemma}[theorem]{Lemma}
\newtheorem{proposition}[theorem]{Proposition}

\theoremstyle{definition}

\theoremstyle{remark}
\newtheorem{remark}[theorem]{Remark}

\newcommand{\norm}[1]{\left\Vert#1\right\Vert}
\newcommand{\abs}[1]{\left\vert#1\right\vert}

\begin{document}
%\begin{frontmatter}
%\nodate
\begin{center}{\bf Nonparametric estimation for fractional\break diffusion processes with random effects}\\
 \small{M. El Omari$^a$,  H. El Maroufy$^a$ and C. Fuchs $^{b,c}$

\small{$^a$ Laboratory of Mathematics and Application, Department of Mathematics, Faculty of sciences and Techniques,  Sultan Moulay Slimane University,  Morocco.}
\\
\small{$^{b}$ Bielefeld University, Faculty of Business Administration and Economics, Bielefeld, Germany.}\\
\small{$^{c}$ Helmholtz Zentrum M\"unchen, German Research Center for Environmental Health GmbH, Institute of Computational Biology, Neuherberg, Germany.}}
%\maketitle
\end{center}
\begin{abstract}
We propose a nonparametric estimation for a class of fractional stochastic differential equations (FSDE) with random effects. We precisely consider general linear fractional stochastic differential equations with drift depending on random effects and non-random diffusion. We build ordinary kernel estimators and histogram estimators and study their $L^p-$risk ($p=1 ~\mbox{or}~ 2$), when $H>1/2$. Asymptotic results are evaluated as both $T=T(N)$ and $N$ tend to infinity.
\end{abstract}

%\end{frontmatter}

%\todo{Add a sentence on why we are doing this. Is there no existing approach?}

%\begin{classcode} 62G86; 60G22; 62G20.\end{classcode}
%%%%%%%%%%%%%%%%%%%%%%%%%%%%%%%%%%%%%%%%%%%%%%%%%%%%%%%%%%%%%%%%%%%%%%%%%%%%
%                                                                          %
%                                INTRODUCTION                              %
%                                                                          %
%%%%%%%%%%%%%%%%%%%%%%%%%%%%%%%%%%%%%%%%%%%%%%%%%%%%%%%%%%%%%%%%%%%%%%%%%%%%
\section{Introduction}\label{Sec.1}
Long-memory processes or stochastic models having long-range dependence phenomena have been paid much attention in view of their applications in a variety of different scientific fields, including (but not limited to)  hydrology \cite{McLeod and hipel 1978}, biology \cite{Collins and De Luca 1994}, medicine \cite{Kuklinski et al. 1989}, economics \cite{Granger 1966} or traffic networks \cite{Willinger et al. 1995}. Perhaps the most popular approach for modeling long memory is the use of fractional Brownian motion (abbreviated as fBm) that has been verified as a good model to describe the long-memory property of some time process. As a consequence, in order to take into account long memory, it is natural to model the data exhibiting long-range dependence by  fBm  instead of the  Brownian motion, known by the independence property of its increments. fBm's have been introduced to the statistics community by Mandelbrot et $al$.\cite{Mandelbort et al. 1968}. A normalized fBm with the Hurst index $H\in (0,1)$ is a centered Gaussian process $\displaystyle \left( W^H_t,~t\geq 0 \right)$ having the covariance $$\mathbb{E}\left( W^H_s W^H_t\right)=\frac{1}{2}\left(t^{2H}+s^{2H}-\abs{t-s}^{2H} \right).$$

Statistical inference for stochastic differential  equations (hereafter SDEs)  driven by fBm  has progressed after the development of stochastic calculus with respect to the fBm. In modeling context, the problems of parameter inference are of particular importance, so the growing number of papers devoted to statistical methods for SDEs with fractional noise is not surprising. We mention  only a few of them; further references can be found in \cite{Mishura 2008, Prakasa 2010}. In \cite{Kleptsyna and Le Breton 2002}, the authors proposed and studied maximum likelihood estimators for fractional Ornstein–Uhlenbeck processes. Related results were obtained in \cite{Prakasa 2003}, where a more general model was considered. In \cite{Hu and Nualart 2010}, the authors proposed a least squares estimator for fractional Ornstein–Uhlenbeck processes and proved its asymptotic normality.\bigskip

\cite{Hu et al. 2011} and \cite{Tudor and Viens 2007} deal with the whole range of Hurst index $H\in (0,1)$, while other papers cited here investigate only the case $H >1/2$ (which corresponds to a long-memory process). Recall that in the case $H =1/2$, we have a classical diffusion, and there is vast literature devoted to it (see, e.g.\cite{Kutoyants 1984}, \cite[Vol II]{Liptser and Shiryaev 2001} and  \cite{Bishwal 2008}, for the review of the topic). \\

 In the context of stochastic differential equation models with random effects (hereafter SDEMRE),  which   are increasingly used in the biomedical field and have proved to be adequate tools for the study of repeated measurements collected on a series of subjects, parametric  inference has recently been investigated by many authors (see e.g.\cite{Ditlevsen and  De Gaetano 2005a, Delattre et al. 2012, Picchini et al. 2010, Picchini and Ditlevsen 2011, Nie and Yang 2005, Nie 2006, Nie 2007,  Antic et al. 2009}).  However, there is no reference at present related  to inference for SDEMREs driven by fBm. The main contribution of this paper is to provide a series of nonparametric estimators of the common density~$f$ of the random effects~$\phi_i$ on $\mathds{R}$ from the observations $\displaystyle X^i(t),~0\leq t\leq T,~i=1,\cdots,N$, which are either kernel estimators or histogram estimators.\\

 We focus on FSDEs
 %\todo{The abbreviation FSDE has not been defined, has it?}
  of the form
  \begin{equation*} \displaystyle dX(t)=\left( a(X(t))+\phi b(t) \right) dt+\sigma(t) W^{H}(t), \end{equation*} where $\phi$ is a random variable with density $f$ belonging to a specified class of functions, and $W^H$ is a normalized fBm with Hurst index $H\in \left(1/2,1\right)$, which may not be known.
 % \todo{Also add how $a, b, c$ are defined, e.g. $a:\mathcal{R}\rightarrow\mathcal{R}$}
   We study the $L^p$-risk ($p=1$ or $2$) of the proposed estimators when $a(\cdot)$ is known or unkown. Asymptotic properties are evaluated as both $T$ and $N$ become large. To our knowledge, this problem has not been investigated in the context of FSDEs with random effects yet. \bigskip

This paper  is organized as follows. In Section \ref{Sec.2}, we introduce the model and some notation. Section 3 is devoted to our main results and is split into two subsections. In Subsection 3.1 we build ordinary kernel estimators and study their $L^2$-risk, while histogram estimators are given with their $L^1$-risk in Subsection~3.2. Section 4 is devoted to numerical simulations. In Section 5, we give concluding remarks. The appendix section provides auxiliary computations and facts which are used in the proof of the main results.\bigskip

 \section{Model and notation}\label{Sec.2}

 Let $\displaystyle\left(\Omega,\mathcal{F},(\mathcal{F}_t^i),\mathbb{P}\right)$ be a stochastic basis satisfying the usual conditions. The natural filtration of a stochastic process is understood as the $\mathbb{P}$-completion of the filtration generated by this process. Let $\displaystyle W^{H,i}=\left( W^{H,i}(t)~,~t\leq T \right)$, $i=1,\cdots,N$ be $N$ independent normalized fractional Brownian motions (fBm) with a common Hurst index $\displaystyle H\in (0,1)$. Let $\displaystyle \phi_1,\cdots,\phi_N$ be $N$ independent and identically distributed (i.i.d) $\mathds{R}$-valued random variables on the common probability space $\displaystyle\left(\Omega,\mathcal{F},\mathbb{P} \right)$ independent of  $\displaystyle \left( W^{H,1},\cdots,W^{H,N}\right)$. Consider $N$ subjects $\displaystyle \left(X^i(t),\mathcal{F}_t^i,t\leq T \right)$ with dynamics ruled by the following general linear stochastic differential equations:

  \begin{eqnarray}\label{Eq1} \displaystyle \nonumber dX^i(t) &=& \left(a(X^i(t))+\phi_i b(t) \right)dt+\sigma(t) dW^{H,i}(t),~t\leq T;\\
   \displaystyle   X^i(0) &=& x^i\in \mathds{R},~~i=1,\cdots,N,
  \end{eqnarray}where $b(\cdot)$ and $\sigma(\cdot)$ are known in their own spaces, but $a(\cdot)$ may be unknown. Let the random effects $\phi_i$ be $ \mathcal{F}_0^i$-measurable with common density $f$ belonging to a specified class of functions for each type of estimators. Sufficient conditions for the existence and uniqueness of solutions to (\ref{Eq1}) can be found in \cite[p. 197]{Mishura 2008} or \cite{Nualart and Rascanu 2002} and references therein.\\

Throughout this paper, we write $u \precsim v $ if an inequality holds up to a non-negative multiplicative constant and $u \varpropto v$ if $u$ equals $v$ up to a non-negative multiplicative constant. We denote by $o(\cdot)$ and $O(\cdot)$ the usual small-oh and big-oh under the probability $\mathbb{P}$, respectively. $\norm{\cdot}$ will denote the $L^2$-norm, unless we specify the norm as $\norm{\cdot}_p$, $p\geq 1$.  $\mathbb{E}$ and $\longrightarrow$ denote the expectation under the $\mathbb{P}$ and the simple convergence, respectively.

\section{Main results}\label{Sec.3}
\subsection{Ordinary kernel density estimators}\label{Subsec.1}
It is well known that standard kernel density estimators for the unknown density $f$ of $\phi_i$ are given by \begin{equation}\label{Eq2}
 \displaystyle \widehat{f}_h(x)=\frac{1}{Nh}\sum_{i=1}^{N}K\left( \frac{x-\phi_i}{h}\right),~~~h>0,
\end{equation}where $K$ is an integrable kernel that has to satisfy some regularity conditions on $f$. The random effects $\phi_i$ are not observed; it is natural to replace them by their estimators and prove the consistency of the proposed kernel estimators. We introduce some statistics which have a central role in the estimation procedure. For $i=1,\cdots,N$, we denote

 \begin{eqnarray*} \displaystyle U^{(1,i)}_t &=& \int_0^t \frac{b(s)}{\sigma^2(s)}dX^i(s),~~ U^{(2)}_t = \int_0^t \frac{b^2(s)}{\sigma^2(s)}ds,\\ \displaystyle  R^{(i)}_t  &=& \int_0^t \frac{a(X^i(s))b(s)}{\sigma^2(s)}ds~~\mbox{and}~~V^{(i)}_t = \int_0^t \frac{b(s)}{\sigma(s)}dW^{H,i}(s).
  \end{eqnarray*}
  We know that $V^{(i)}=\left( V^{(i)}_t,~t\geq 0\right)$, $i=1,\cdots,N$ are Wiener integrals with respect to fBm. A sufficient condition (see \cite{Prakasa 2010,Mishura 2008}) for the integrals $V^{(i)}$ to be well-defined is that $b(\cdot)/\sigma (\cdot)\in L^2(\mathds{R}_+)\cap L^1(\mathds{R}_+)$. The following assumptions are needed to estimate the random effects $\phi_i$: \begin{itemize} \item[\bf{A}$_1$ :] There exist $c_0,c_1 >0$ such that $$\displaystyle c_0^2\leq \frac{b^2(s)}{\sigma^2(s)}\leq
    c_1^2,~~ \mbox{for all}~ s\in\mathds{R}_+.$$
\item[\bf{A}$_2$ :] For $i=1,\cdots,N$, $\displaystyle
    M_i:=\mathbb{E}\left(\int_0^T\frac{a(X^i(s))}{\sigma^2(s)}ds\right)^2<\infty$.
\end{itemize}
%\noindent
 \begin{proposition}\label{PRO1}
  Let the assumptions \textbf{A}$_1$-\textbf{A}$_2$ be fulfilled. For $i=1,\cdots,N$ and $H>1/2$, we have
  $$\displaystyle \mathbb{E}\abs{\widehat{\phi}_{i,T}-\phi_i}^2\longrightarrow 0~\mbox{as}~T\rightarrow \infty,  ~\mbox{where}~ \displaystyle \widehat{\phi}_{i,T} :=U^{(1,i)}_T/U^{(2)}_T .$$
   \end{proposition}

\begin{proof}\label{proof.PRO1} Equation (\ref{Eq1}) yields \begin{equation*} \displaystyle U^{(1,i)}_t = R^{(i)}_t+\phi_i U^{(2)}_t+V^{(i)}_t,~t\leq T,~i=1,\cdots,N. \end{equation*}Thus \begin{equation}\label{Eq3} \displaystyle \frac{1}{2} \mathbb{E}\abs{\widehat{\phi}_{i,T}-\phi_i}^2 \leq \mathbb{E}\left(\frac{R^{(i)}_T}{U^{(2)}_T}\right)^2+\mathbb{E}\left(\frac{V^{(i)}_T}{U^{(2)}_T}\right)^2. \end{equation}We shall show that the expectations on the right hand side in (\ref{Eq3}) vanish as $T$ tends to infinity. Applying results in \cite[Corollary 1.92]{Mishura 2008} and the Jensen inequality, respectively, we obtain \begin{eqnarray*} \displaystyle \mathbb{E}\left(\frac{V^{(1)}_T}{U^{(2)}_T}\right)^2 &=&  \frac{1}{c_0^4T^2} \mathbb{E}\left( \int_0^T \frac{b(s)}{\sigma(s)} dW^{H,1}(s)\right)^2\\ \displaystyle &=& \frac{C_H^2c_1^2}{c_0^4T^2}\left( \int_0^T \abs{\frac{b(s)}{\sigma(s)}}^{1/H} ds\right)^{2H}\\ \displaystyle &=& \frac{C_H^2c_1^2}{c_0^4T^{2-H}}\longrightarrow 0~\mbox{as}~T\rightarrow \infty, \end{eqnarray*}where $C_H$ is a non-negative constant due to the Hardy-Littlewood theorem (see, \cite{Mishura 2008}). Using the fact that $\displaystyle \abs{uv}\leq \frac{1}{2}\left(\varepsilon u^2+\frac{v^2}{\varepsilon} \right)$ for all $u,v\in\mathds{R}$ and $\varepsilon>0$, we have \begin{eqnarray*} \displaystyle \mathbb{E}\left(\frac{R^{(1)}_T}{U^{(2)}_T}\right)^2 &=&  \frac{1}{4} \mathbb{E}\left\lbrace \frac{\varepsilon \int_0^T a^2(X^1(s))/\sigma^2(s) ds +\varepsilon^{-1}U^{(2)}_T}{U^{(2)}_T} \right\rbrace^2\\ \displaystyle &=& \frac{1}{2}\left\lbrace \frac{1}{\varepsilon^2}+\frac{\varepsilon^2}{c_0^4T^2}\mathbb{E}\left(\int_0^\infty a^2(X^1(s))/\sigma^2(s) ds \right)^2\right\rbrace. \end{eqnarray*}By choosing $\varepsilon=\sqrt{T}$, we get the desired result and the proof of Proposition \ref{PRO1} is complete. \end{proof} Now, substituting $\phi_i$ by its estimator $\widehat{\phi}_{i,T}$ in (\ref{Eq2}), we obtain the kernel estimators \begin{equation}\label{Eq4} \displaystyle \widehat{f}^{(1)}_h(x)=\frac{1}{Nh}\sum_{i=1}^{N}K\left( \frac{x-\widehat{\phi}_{i,T}}{h}\right). \end{equation} \begin{proposition}\label{PRO2} Consider Equation (\ref{Eq1}) where $a(\cdot)$ is unknown and consider the estimator $\widehat{f}^{(1)}_h$ given by (\ref{Eq4}). Assume that \textbf{A}$_1$ and \textbf{A}$_2$ are satisfied. If the kernel $K$ is differentiable with $\displaystyle \norm{K}^2+\norm{K'}^2<\infty$, then \begin{equation*} \displaystyle \mathbb{E}\norm{\widehat{f}^{(1)}_h-f}^2\leq 2\norm{f_h-f}^2+\frac{\norm{K}^2}{Nh}+\frac{\norm{K'}^2}{Th^3}\left(1+\frac{M_1}{c_0^4}+ \frac{2C_H^2c_1^2}{c_0^4T^{1-H}} \right), \end{equation*}where $\displaystyle f_h(x):=K_h*f(x)=\int_{\mathds{R}} K(x-u)f(u)du$. \end{proposition} \begin{proof}\label{proof.PRO2} Simple computations show that \begin{eqnarray}\label{Eq5} \nonumber \displaystyle \mathbb{E}\norm{\widehat{f}^{(1)}_h-f}^2 &=& \norm{f-\mathbb{E}(\widehat{f}^{(1)}_h)}^2+\mathbb{E}\left(\norm{\widehat{f}^{(1)}_h-\mathbb{E}(\widehat{f}^{(1)}_h)}^2\right) \\ \displaystyle &\leq &   2\norm{f-f_h}^2+2\norm{f_h-\mathbb{E}(\widehat{f}^{(1)}_h)}^2+\mathbb{E}\left(\norm{\widehat{f}^{(1)}_h-\mathbb{E}(\widehat{f}^{(1)}_h)}^2\right). \end{eqnarray}To complete the proof, we evaluate the last two terms in (\ref{Eq5}). Set $\displaystyle\eta_{i,T}(x)=K_h(x-\widehat{\phi}_{i,T})-\mathbb{E}\left( K_h(x-\widehat{\phi}_{i,T})\right)$, where $\displaystyle K_h(u)=\frac{1}{h}K\left( \frac{u}{h}\right)$. $\eta_{i,T}(x)$, $i=1,\cdots,N$ are i.i.d random variables with $\mathbb{E}\left[\eta_{1,T}(x)\right]=0$, and with a change of variables $\displaystyle \frac{x-\widehat{\phi}_{1,T}}{h}=y $ in the second inequality below, we get

 \begin{eqnarray*} \displaystyle \int_{\mathds{R}} \mathbb{E}\left(\eta_{1,T}(x)\right)^2 dx &=&\int_{\mathds{R}} \mbox{Var}\left( K_h(x-\widehat{\phi}_{1,T})\right)dx \\ \displaystyle &\leq  &\int_{\mathds{R}} \mathbb{E}\left( K_h(x-\widehat{\phi}_{1,T})\right)^2dx \\ \displaystyle &\leq  &\frac{1}{h^2}\mathbb{E}\int_{\mathds{R}} \left( K\left(\frac{x-\widehat{\phi}_{1,T}}{h}\right)\right)^2dx \\
 \displaystyle &\leq  &\frac{1}{h}\int_{\mathds{R}} K^2(y)dy.
  \end{eqnarray*} Thus
\begin{eqnarray*} \displaystyle \mathbb{E}\left(\norm{\widehat{f}^{(1)}_h-\mathbb{E}(\widehat{f}^{(1)}_h)}^2\right) &=& \mathbb{E}\int_{\mathds{R}} \left(\widehat{f}^{(1)}_h(x)-\mathbb{E}\widehat{f}^{(1)}_h(x) \right)^2 dx\\ \displaystyle &=& \frac{1}{N^2}\mathbb{E}\int_{\mathds{R}} \left(\sum_{i=1}^{N}\eta_{i,T}(x) \right)^2 dx\\ \displaystyle &=& \frac{1}{N}\int_{\mathds{R}} \mathbb{E}\left(\eta_{1,T}(x)\right)^2 dx \leq  \frac{\norm{K}^2}{Nh}. \end{eqnarray*}
There remains to find an upper bound of the middle term in (\ref{Eq5}). First, note that $\displaystyle f_h(x)=\int_{\mathds{R}} f(y)K_h(x-y)dy=\mathbb{E}\left( K_h(x-\phi_1)\right)$. Taylor's theorem with integral remainder yields \begin{equation*} \displaystyle K_h(x-\widehat{\phi}_{1,T})-K_h(x-\phi_1) =  \frac{(\phi_1-\widehat{\phi}_{1,T})}{h^2}\int_0^1 K'\left(\frac{1}{h}(x-\phi_1+u(\phi_1-\widehat{\phi}_{1,T})) \right) du. \end{equation*}
 Now, set $\displaystyle g(x,u)=K'\left(\frac{1}{h}(x-\phi_1+u(\phi_1-\widehat{\phi}_{1,T})) \right)$, then%\todo{put integral bounds everywhere (8 times until/incl. Cor. 3.3)}

\begin{eqnarray*}
\displaystyle \nonumber \norm{f_h-\mathbb{E}(\widehat{f}^{(1)}_h)}^2 &=& \int_{\mathds{R}} \left[ \mathbb{E}\left( K_h(x-\widehat{\phi}_{1,T})-K_h(x-\phi_1)\right)\right]^2dx \\
 \displaystyle &\leq &\int_{\mathds{R}} \mathbb{E}\left( K_h(x-\widehat{\phi}_{1,T})-K_h(x-\phi_1)\right)^2dx \\
  \displaystyle &\leq & \mathbb{E}\left[ \frac{(\phi_1-\widehat{\phi}_{1,T})^2}{h^4}\int_{\mathds{R}}\left[\int_0^1 g(x,u) du\right]^2dx\right] \\
  \displaystyle &\leq &\mathbb{E}\left[ \frac{(\phi_1-\widehat{\phi}_{1,T})^2}{h^4}\left[\int_0^1\left(\int_{\mathds{R}} g^2(x,u) dx\right)^{1/2}du\right]^2\right].
 \end{eqnarray*}
 The last inequality given above is justified by the generalized Minkowski inequality (see \cite[Lemma A.1]{Tsybakov 2009}). By change of variables $\displaystyle y=\frac{1}{h}\left(x-\phi_1+u(\phi_1-\widehat{\phi}_{1,T}) \right)$, we get $\displaystyle \int_{\mathds{R}} g^2(x,u) dx = \norm{K'}^2 h$. Thus $\displaystyle \norm{f_h-\mathbb{E}(\widehat{f}^{(1)}_h)}^2 \leq \frac{\norm{K'}^2 }{h^3}\mathbb{E}(\phi_1-\widehat{\phi}_{1,T})^2$, which completes the proof (see the proof of Proposition \ref{PRO1}).
  \end{proof}
   We recall that a kernel of order $l\geq 1$ (for the construction of such a kernel we refer to \cite[p.10]{Tsybakov 2009}) satisfies $\displaystyle \int_{\mathds{R}} K(u)du=1$ and $\displaystyle \int_{\mathds{R}} u^jK(u)du=0$, for $j=1,\cdots,l$. For constants $\beta>0$ and $L>0$, we define the Nikol'ski class $\mathcal{N}(\beta,L)$ as the set of functions $\displaystyle f : \mathds{R}\longrightarrow\mathds{R}$, whose derivatives $f^{(l)}$ of order $l=\lfloor \beta\rfloor$ exist and satisfy $$\displaystyle \left[\int_{\mathds{R}} \left( f^{(l)}(x+t)-f^{(l)}(x)\right)^2dx\right]^{1/2} \leq L\abs{t}^{\beta-l},~\forall t\in\mathds{R} ,$$
   where $\lfloor \beta\rfloor$ denotes the greatest integer strictly
  % \todo{Why strictly? Then $\lfloor 5\rfloor=4$, but this is not what you want, is it? : It does not matter see the proof of Proposition 1.5 in Tsybakov}
   less than the real number $\beta$.
\begin{corollary}\label{COR1} Assume that $f\in\mathcal{N}(\beta,L)$ and that the kernel $K$ has order $l=\lfloor\beta\rfloor$ with $\displaystyle \int_{\mathds{R}} \abs{u}^\beta \abs{K(u)} du<\infty$. Fix $\alpha>0$ and take $h=\alpha N^{-1/(2\beta+1)}$ and $T\geq N^{(2\beta+3)/(2\beta+1)}$. Then for any $N\geq 1$, the kernel estimator $\widehat{f}^{(1)}_h$ satisfies $\displaystyle \mathbb{E}\norm{\widehat{f}^{(1)}_h-f}^2 \lesssim N^{-2\beta/(2\beta+1)}$. \end{corollary} \begin{corollary}\label{COR2} Consider Equation (\ref{Eq1}) where $a(\cdot)$ is known. We introduce the estimators \begin{equation*}\label{Eq6} \displaystyle \widehat{f}^{(2)}_h(x)=\frac{1}{Nh}\sum_{i=1}^{N}K\left( \frac{x-\widetilde{\phi}_{i,T}}{h}\right), \end{equation*}where $\displaystyle \widetilde{\phi}_{i,T}:=\widehat{\phi}_{i,T}-R^{(i)}_T/U^{(2)}_T$. Under the assumption \textbf{A}$_1$, the estimators $\widehat{f}^{(2)}_h$ are consistent with the same optimal rate as for $\widehat{f}^{(1)}_h$. \end{corollary} \begin{remark} The assumption \textbf{A}$_2$ can be weakened as follows \begin{itemize} \item[\bf{A'}$_2$ :] For each $i$, there exists $\delta>0$ such that
    $$\limsup_{t\rightarrow\infty}\frac{1}{t^{2-\delta}\log(t)}\mathbb{E}\left(\int_0^t\frac{a(X^i(s))}{\sigma^2(s)}ds\right)^2<\infty.$$
\end{itemize} \end{remark}

\subsection{Histogram estimators}\label{Subsec.2} Consider a sequence of partitions of $\mathds{R}$ of the form $\displaystyle \mathcal{P}_N=\left\lbrace A_{Nj},~j=1,2,\cdots \right\rbrace$, $N\geq 1$, where all $A_{Nj}$'s are Borel sets with finite nonzero Lebesgue measure. We assume that the sequence of partitions is rich enough such that the class of Borel sets $\mathcal{B}$ is equal to
$$ \displaystyle\bigcap_{N=1}^{\infty}\bm{\sigma}\left( \displaystyle\bigcup_{m=N}^{\infty}\mathcal{P}_m\right),$$
where we use the symbol $\bm{\sigma}$ here for the $\sigma$-algebra generated by a class of sets.\bigskip

Given a sequence of i.i.d random variables $X_1,\cdots,X_N$, with common density $f$, the histogram estimate is (as in \cite{Devroye and Gyorfi 1985}) defined by
\begin{equation*} \displaystyle T(X_\cdot)(x)=\frac{1}{N}\sum_{i=1}^{N}\frac{\chi_{(X_i\in A_{Nj})}}{\lambda(A_{Nj})},~~ x\in A_{Nj},
\end{equation*}where $\lambda$ denotes the Lebesgue measure. For our case, we will consider the following histogram estimators $\displaystyle  \widehat{f}_h^{(3)}(x)= T(\widehat{\phi}_\cdot)(x)$; $\displaystyle  \widehat{f}_h^{(4)}(x)= T(\widetilde{\phi}_\cdot)(x)$. If the density $f$ of the random effects $\phi_i$ has compact support, then a good estimator should have
%\todo{Why so vague?}
compact support as well. To guarantee such property we trim the proposed estimators by $\displaystyle \chi_{\mbox{supp}f}$. \bigskip

Let $\mathcal{F'}_b$ denote the class of functions satisfying \begin{itemize} \item[(i)] $f$ is absolutely continuous with derivative $f'$ (almost everywhere); \item[(ii)] $f'$ is bounded and continuous ( $\displaystyle\int_{\mathds{R}} \abs{f'}<\infty$).
%\todo{add integral bounds}
 \end{itemize} We consider the partitions $A_{Nj}=\left[hj,h(j+1)\right)$, $j\in\mathbb{Z}$. The following special functions will be used later: $\displaystyle r_N(x)=\frac{x}{h}-j$, $\displaystyle z_N(x)=\left(1-2 r_N(x) \right)f'(x)$ and $$\displaystyle \Psi(u)=\sqrt{\frac{2}{\pi}}\left( u\int_0^u e^{-x^2/2}dx+e^{-u^2/2}\right),~~u\geq 0.$$
\begin{proposition}\label{PRO4} Let $f\in\mathcal{F'}_b$ have compact support $A$ and assume that $1,\cdots,J $ are nonzero indices for which $\displaystyle \lambda\left( A_{Nj}\cap A \right) \neq 0$ and $T=T(N)$, where $\lambda$ is the Lebesgue measure. Then, the following statements hold true:
 \begin{itemize}
\item[\rm{(i)}] When $a(\cdot)$ is unknown, under the assumptions \textbf{A}$_1$ and \textbf{A}$_2$, we have \begin{equation*} \displaystyle \mathbb{E}\norm{\widehat{f}_h^{(3)}-f}_1\leq \psi_1(N,h)+\psi_2(h)+\frac{dJ}{h^2\sqrt{T}}+o\left(h+\frac{1}{\sqrt{Nh}} \right),
\end{equation*} where $d$ is some non-negative constant and
 \begin{eqnarray*} \displaystyle  \psi_1(N,h) \hspace{-0.3cm}&=&\hspace{-0.3cm} \int_{\mathds{R}} \sqrt{\frac{f}{Nh}}\Psi\left( \frac{h}{2}\abs{z_N}\sqrt{\frac{Nh}{f}}\right)\rightarrow 0\mbox{ as}~h\rightarrow 0,~Nh\rightarrow \infty,\\ \displaystyle \psi_2(h) \hspace{-0.3cm}&=& \hspace{-0.3cm}\frac{2}{N}\sum_{i=1}^{N}\sum_{j=1}^{J}\mathbb{P}\left( \phi_i\in A_{Nj} \right)^{1/2}\rightarrow 0~\mbox{as}~h\rightarrow 0~\mbox{(see Lemma \ref{LEM3}).}
\end{eqnarray*}
\item[\rm{(ii)}] When $a(\cdot)$ is known, we may relax the assumption \textbf{A}$_2$, and the same result holds for
    $\widehat{f}_h^{(4)}$.
\end{itemize}
\end{proposition} \begin{proof}By virtue of \cite[Theorem 6]{Devroye and Gyorfi 1985}, and for sufficiently small $h$ such that $Nh\rightarrow \infty$, we have
 \begin{eqnarray*}
 \displaystyle \mathbb{E}\norm{\widehat{f}_h^{(3)}-f}_1 &\leq  & \sum_{j} \mathbb{E}\int_{A_{Nj}\cap A} \abs{\widehat{f}_h^{(3)}(x)-f(x)} dx\\
  \displaystyle &\leq  &\sum_{j} \mathbb{E}\int_{A_{Nj}} \abs{T(\phi_\cdot)(x)-f(x)} dx \displaystyle  + \sum_{j\leq J} \mathbb{E}\int_{A_{Nj}} \abs{T(\widehat{\phi}_\cdot)(x)-T(\phi_\cdot)(x)} dx \\
   \displaystyle &\leq & \mathbb{E}\norm{T(\widehat{\phi}_\cdot)-f}_1 \displaystyle  + \frac{1}{Nh}\sum_{j=1}^{ J}\sum_{i=1}^{ N}  \int_{hj}^{h(j+1)}  \mathbb{E}\abs{\chi_{(\widehat{\phi}_{i,T}\in A_{Nj})}-\chi_{(\phi_i\in A_{Nj})} } dx  \\
   \displaystyle &\leq &  \psi_1(N,h)+o\left(h+\frac{1}{\sqrt{Nh}} \right)\\
   \displaystyle & & ~~~~+\frac{1}{Nh}\sum_{j=1}^{ J}\sum_{i=1}^{ N}  \int_{hj}^{h(j+1)}  \mathbb{E}\abs{\chi_{(\widehat{\phi}_{i,T}\in A_{Nj})}-\chi_{(\phi_i\in A_{Nj})} } dx.
    \end{eqnarray*}
    Let $\nu(N,J,h)$ denote the last term in the last inequality above. The sequence $\widehat{\phi}_{i,T(N)}$ converges weakly to $\phi_i$, since it converges in $L^2$-sense  as $N$ tends to infinity (say $T(N)\rightarrow \infty$).
    %\todo{Why $\nearrow$ here whilst you are using $\rightarrow$ throughout the paper?}
     Thus, by using Lemma \ref{LEM2}, we obtain
     \begin{equation*}
      \displaystyle \nu(N,J,h) \leq \frac{\sqrt{2}}{N}\sum_{j=1}^{ J}\sum_{i=1}^{ N}  \mathbb{P}(\phi_i\in A_{Nj})^{1/2}\displaystyle \left[\mathbb{P}(\widehat{\phi}_{i,T}\notin A_{Nj})^{1/2}+\mathbb{P}(\phi_i\notin A_{Nj})^{1/2}\right].
      \end{equation*}Let $\alpha\in (0,1)$ to be specified later. We apply Lemma \ref{LEM1} to get
     \begin{eqnarray*} \displaystyle \mathbb{P}\left(\widehat{\phi}_{i,T}\notin A_{Nj}\right) &\leq & \mathbb{P}\left( \abs{\widehat{\phi}_{i,T}-h(j+1/2)}\geq h/2\right)\\ \displaystyle &\leq &\mathbb{P}\left( \abs{\widehat{\phi}_{i,T}-\phi_i}\geq (1-\alpha)h/2\right)+\mathbb{P}\left( \abs{\phi_i-h(j+1/2)}\geq \alpha h/2\right)\\ \displaystyle &\leq &\frac{4\mathbb{E}\left( \widehat{\phi}_{i,T}-\phi_i\right)^2}{(1-\alpha)^2h^2}+\mathbb{P}\left( \phi_i\notin A^{(\alpha)}_{Nj}\right)\leq \frac{d_1}{(1-\alpha)^2h^2T}+1,
     \end{eqnarray*}
     where $d_1$ is some non-negative constant (see the proof of Proposition \ref{PRO1}) and $\displaystyle A^{(\alpha)}_{Nj}=\left( h(j+\frac{1-\alpha}{2}),h(j+\frac{1+\alpha}{2})\right)$. Similarly, one can prove that $\displaystyle \mathbb{P}\left(\phi_i\notin A_{Nj}\right) \leq  \frac{d_1}{(1-\alpha)^2h^2T}+1$. Thus \begin{eqnarray*} \displaystyle \nu(N,J,h) &\leq & \frac{2}{N}\sum_{j=1}^{ J}\sum_{i=1}^{ N} \left[ \mathbb{P}\left(\phi_i\in A_{Nj}\right)\left(\frac{d_1}{(1-\alpha)^2h^2T}+1\right)\right]^{1/2}\\ \displaystyle &\leq &\frac{2 \sqrt{ d_1} J}{(1-\alpha)h\sqrt{T}}+\frac{2}{N}\sum_{j=1}^{ J}\sum_{i=1}^{ N} \left[\mathbb{P}\left(\phi_i\in A_{Nj}\right)\right]^{1/2}, \end{eqnarray*}where we used the fact that $\displaystyle \sqrt{u+v}\leq\sqrt{u}+\sqrt{v}$, for all $u,v\in \mathds{R}_{+}$. Set $d=2\sqrt{d_1}$ and $\alpha=1-h$ to complete the proof.
 \end{proof}

\begin{proposition}\label{PRO5} We have \begin{equation}\label{Q1} \displaystyle \psi_2(h) = O(h^\delta),~~\mbox{where}~\delta\in (0,1/2). \end{equation}
\end{proposition}
\begin{proof}\label{proof.PRO5} Let $\displaystyle \delta,\delta^*\in (0,1)$ such that $\delta+\delta^*=1$. It is easy to see that \begin{eqnarray*} \displaystyle  \mathbb{P}\left(\phi_i\in A_{Nj} \right) &=& \mathbb{P}\left(\phi_i\in A_{Nj} \right)^{\delta^*} \left(\int_{hj}^{h(j+1)} f(t)dt\right)^{\delta}\\ \displaystyle  &\leq & \left[\sup_{i,j}\mathbb{P}\left(\phi_i\in A_{Nj} \right)\right]^{\delta^*}\sup_{t}f(t)^{\delta}h^\delta\\ \displaystyle  &\leq & \frac{e^{-j\delta^*}}{j!}\sup_{t}f(t)^{\delta}h^\delta,~~h\in (0,h_0), \end{eqnarray*}where $h_0$ is some non-negative number independent of $i$ and $j$. Thereby, \begin{equation*} \displaystyle \psi_2(h) = \sup_{t}f(t)^{\delta/2}h^{\delta/2} \sum_{j\geq 1}\frac{e^{-j\delta^*/2}}{\sqrt{j!}}<\infty. \end{equation*} \end{proof} Let $\displaystyle T=T(N)\geq J^4$ so that \begin{equation}\label{Q2} \displaystyle \frac{dJ}{h^2\sqrt[4]{T}} = O\left(h^{\delta'} \right),~\mbox{and set }~  h\propto N^{-\delta''}. \end{equation} As mentioned in \cite[Theorem 6]{Devroye and Gyorfi 1985}, \begin{equation}\label{Q3} \displaystyle \psi_1(N,h)+o\left( h+\frac{1}{\sqrt{Nh}}\right) = O\left(N^{-1/3} \right). \end{equation} Fitting rates of convergence given in (\ref{Q1}),(\ref{Q2}) and (\ref{Q3}), we choose $\delta'=\delta$, $\delta''=1/(3\delta)$. An arbitrary choice of $\delta$ may violate the crucial condition $Nh\longrightarrow \infty$ as $h\rightarrow 0$. Choosing $\displaystyle \delta\in \left(1/3,1/2 \right)$, we guarantee that all conditions on $T$, $J$, $N$ and $h$ are fulfilled. Finally $\displaystyle \mathbb{E}\norm{\widehat{f}_h^{(3)}-f}_1 = O\left(N^{-1/3} \right)$. In a similar fashion, we can prove that $\widehat{f}_h^{(4)}$ as well as  $\widehat{f}_h^{(3)}$ have the same rates of convergence .

 \section{Numerical simulation}\label{Sec.3.7} As an example, we consider the following Langevin equation as dynamics of the subject $X^i$: \begin{eqnarray}\label{EXX} \displaystyle dX^i(t) &=& \left( -\lambda X^i(t)+\phi_i b(t)\right)dt+\sigma dW^{H,i}(t),~t\leq T\\ \nonumber\displaystyle X^i(0) &=& x^i\in\mathds{R}, \end{eqnarray}
 where $H>1/2$, $\lambda,~\sigma >0$ and $\phi_i$ is a random variable such that $\mathbb{E}\abs{\phi_i}^4<\infty$, $i=1,\ldots,N$. Assume that $b_1^2\leq b(t)^2\leq b_2^2$, for all $t\leq T$. The common density $f$ of $\phi_i$ can be estimated by $\widehat{f}_h^{(1)}$ and $\widehat{f}_h^{(2)}$, since both \textbf{A}$_1$ and \textbf{A'}$_2$ hold true. We shall only show that \textbf{A'}$_2$ holds:\bigskip

First, note that $\displaystyle X^i(t)=e^{-\lambda t}\left(x^i+\phi_i\int_0^t e^{\lambda s} b(s)ds +\sigma\int_0^t e^{\lambda s}dW^{H,i}(s) \right)$ is a solution to Equation (\ref{EXX}). Set $\displaystyle X^{i,H}_\xi (t):=\sigma\int_{-\infty}^t e^{-\lambda (t-s)}dW^{H,i}(s)$. The process $X^{i,H}_\xi$ solves the equation $dX^i(t)=-\lambda X^i(t) dt+\sigma dW^{H,i}(t)$, with initial value $\xi=X^{i,H}_\xi (0)$. It is clear that $X^{i,H}_\xi$ is a stationary Gaussian process. From \cite[Theorem 2.3]{Cheridito et al 2003}, it follows that $X^{i,H}_\xi$ is ergodic. Therefore, the ergodic theorem implies that \begin{equation*} \displaystyle \frac{1}{T}\int_0^T X^{i,H}_\xi (t)^4 dt \longrightarrow \mathbb{E}\xi^4~~\mbox{as}~T\rightarrow \infty \end{equation*} almost surely and in $L^2$. Set $Z^i(t)=X^i(t)-\phi_i \int_0^t e^{-\lambda (t-s)}b(s)ds $. Simple computations show that $ \displaystyle Z^i(t)=e^{-\lambda t}(x^i-\xi)+X^{i,H}_\xi (t) $. Hence
 \begin{eqnarray*}
 \displaystyle \frac{1}{T}\mathbb{E}\left( \int_0^T \left(-\lambda X^i(t)/\sigma \right)^2 dt\right)^ 2  ~~~~~~~~~~~~~~~~~~~~~~~~~~~~~~~~~~~~~~~~~~~~~~~~~~~~~~~~~~~~~~~~~~~~~~&~&\\
  \displaystyle \leq \frac{8\lambda^4}{\sigma^4 T}\left\lbrace \mathbb{E}\int_0^T Z^i(t)^4 dt+\int_0^T \mathbb{E}\left(\phi_i\int_0^t e^{-\lambda (t-s)}b(s)ds \right)^4 dt \right\rbrace ~~~~~~~~~~~~~~~~~~~~~~~~~~~~~~~~~~&~&\\
  \displaystyle \leq \frac{64\lambda^4}{\sigma^4 T}\mathbb{E}\int_0^T X^{i,H}_\xi (t)^4dt+\frac{64\lambda^4}{\sigma^4 T}\int_0^T e^{-4\lambda t}dt\mathbb{E}(x^i-\xi)^4+\frac{8\lambda^4 b_1^4}{\sigma^4 T}\int_0^T\int_0^t e^{-4\lambda(t-s)}ds dt \mathbb{E}\phi_i^4 ~&~&\\
 \displaystyle \leq \frac{64\lambda^4}{\sigma^4 }\mathbb{E}\left\lbrace \frac{1}{T}\int_0^T X^{i,H}_\xi (t)^4 dt\right\rbrace+ \frac{16\lambda^3}{\sigma^4 T}\mathbb{E}(x^i-\xi)^4+\frac{2\lambda^3 b_1^4}{\sigma^4} \mathbb{E}\phi_i^4 ~~~~~~~~~~~~~~~~~~~~~~~~~~~~~~~~~~~~&~&\\
  \displaystyle \longrightarrow \frac{64\lambda^4}{\sigma^4 } \mathbb{E}\xi^4+\frac{2\lambda^3 b_1^4}{\sigma^4} \mathbb{E}\phi_i^4~\mbox{as}~T\rightarrow \infty, ~~~~~~~~~~~~~~~~~~~~~~~~~~~~~~~~~~~~~~~~~&~&
 \end{eqnarray*}which in turn implies that
 \begin{equation*}
  \displaystyle \lim_{T\rightarrow\infty} \frac{1}{T^{2-\delta}\log (T)}\mathbb{E}\left( \int_0^T \left(-\lambda X^i(t)/\sigma \right)^2 dt\right)^2=0,~\mbox{for all}~\delta\in (0,1). \end{equation*}

 % \todo[color=blue!20]{I rephrased the blue-colored paragraphs. Please check whether they is correct.}

 % {\color{blue}
  For illustration, we simulate model (\ref{EXX}) with $b(t)=\sigma=1$, estimate the densities of the random effects and compare these to the true data-generating density. In detail, we use up to 25 exact simulations with $\lambda=3\times 10^{-3}$, $x^i=0$, $N=1000$ and $T=100;10$.
 % \todo[color=blue!20]{I corrected from $T=1$ to $T=100$. Is that correct?: Yes (see the first Proposition)}
   The random effects are Gaussian distributed, $\mathcal{N}(1,0.8)$, and Gamma distributed, $\Gamma(2,0.9)$, where $2$ is the shape parameter and $0.9$ the scale parameter. Figures~\ref{Fig1}, \ref{Fig2}, \ref{Fig3} and~\ref{Fig4} display the estimates~$\widehat{f}_h^{(1)}$ and~$\widehat{f}_h^{(3)}$ for different values of the Hurst index, $H\in\{0.25,0.75,0.85\}$ and $T=100;10$. Improving the accuracy of our estimators requires that both $N$ and $T$ be sufficiently large. However, for $T$ being only moderately large ( say $T=10$ ) and/or $H<1/2$ (which is not supported by our theoritical framework), the estimated curves match the theoretical curves satisfyingly well. The estimators $\widehat{f}_h^{(2)}$ and $\widehat{f}_h^{(4)}$ lead to similar results, thus we omitted them.
  % \todo{Did you generate the plots for f2 and f4? Should we provide these as additional material? We need some kind of confirmation.}
    However, for the current example where $a(\cdot)$ is known, $\widehat{f}_h^{(1)}$ and $\widehat{f}_h^{(3)}$ are recommended:
    For~$\widehat{f}_h^{(2)}$ and $\widehat{f}_h^{(4)}$, we may relax the assumptions \textbf{A}$_2$ and \textbf{A'}$_2$, but the results are more time-consuming as we need to compute
     $\widehat{\phi}_{i,T}$ and ${R_T^{(i)}}/{U_T^{(2)}}$, while $\widehat{f}_h^{(1)}$ and $\widehat{f}_h^{(3)}$ require only $\widehat{\phi}_{i,T}$.
    %\todo[inline,color=yellow]{I think we should also provide the Matlab code in a public repository and a readme file on how to reproduce the graphics. My student could do this and put it in our Github account. Should we? Or do you prefer a link that points to your pages?}
%%%%%%%%%%%%%%%%%%%%%%%%%%%%%%%%%%%%%%%%%%%%%%%%%%%%%%%%%%%%%%%%%%%%%%%%%%%%%%%%%%%%%%%%%%%%%%%%%%%%%%%%%%%%%%%%%
%%%%%%%%%%%%%%%%%%%%%%%%%%%% T=100 %%%%%%%%%%%%%%%%%%%%%%%%%%%%%%%%%%%
\begin{figure}
  \begin{center}
  \begin{tabular}{cccc}
  & {\bf $H=0.25$} & {\bf $H=0.75$} & {\bf $H=0.85$}\\
  \rotatebox{90}{\parbox[c]{67mm}{\hspace*{10mm} \bf Gaussian random effects}}
    &  \includegraphics[scale=0.95,totalheight=7cm,width=0.36\textwidth]{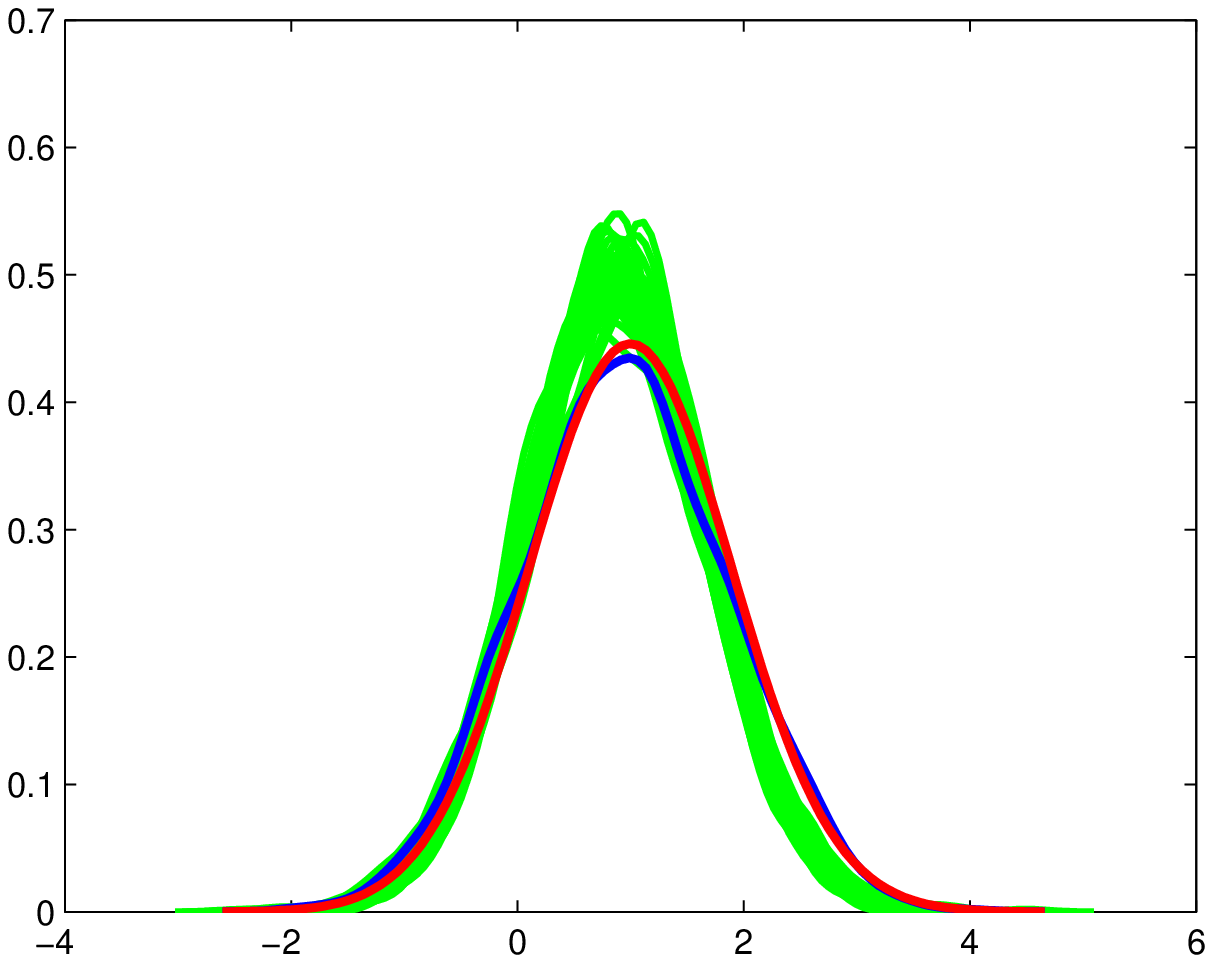}
   &  \hspace{-0.7cm}\includegraphics[scale=0.95,totalheight=7cm,width=0.36\textwidth]{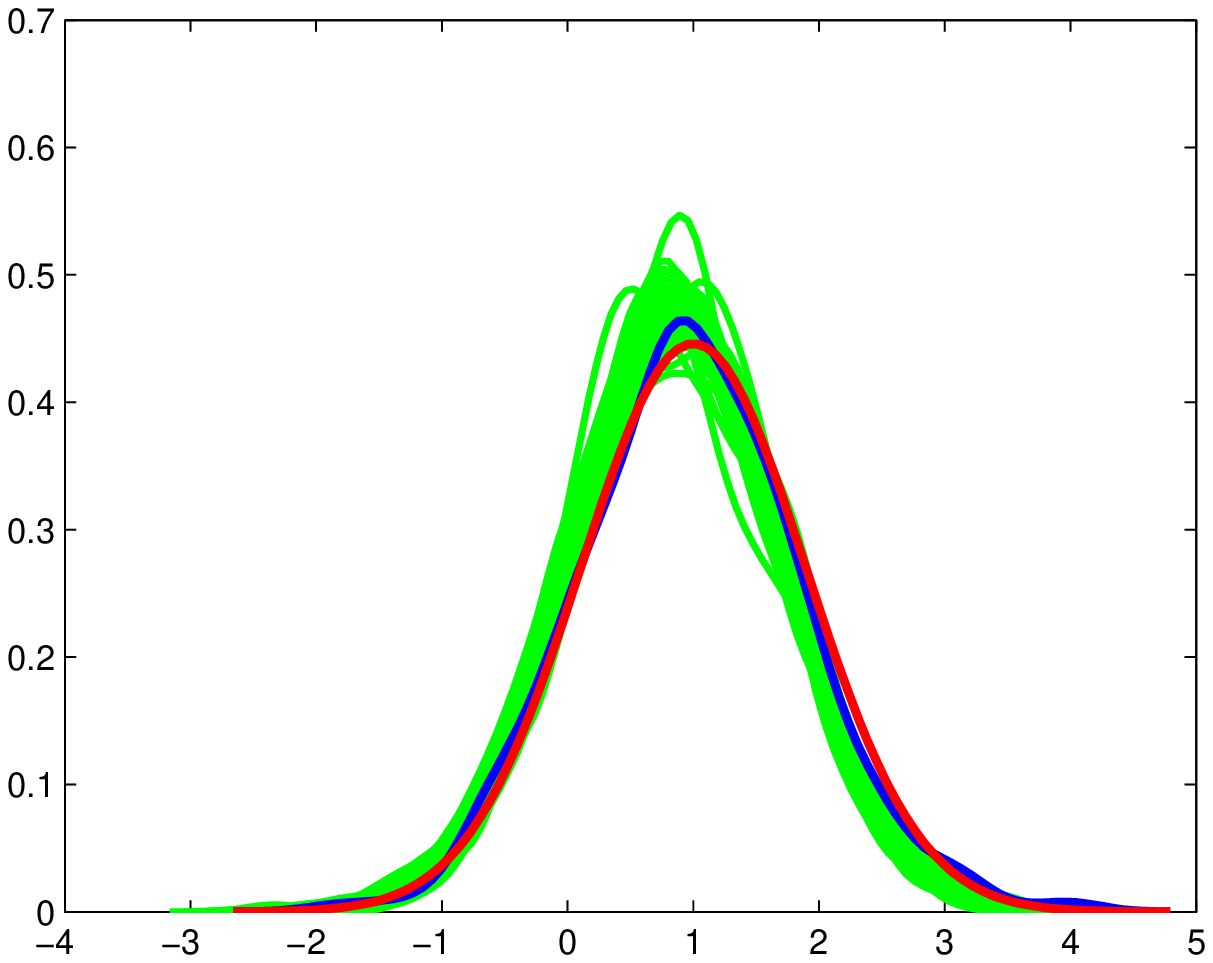}
    & \hspace{-0.7cm} \includegraphics[scale=0.95,totalheight=7cm,width=0.36\textwidth]{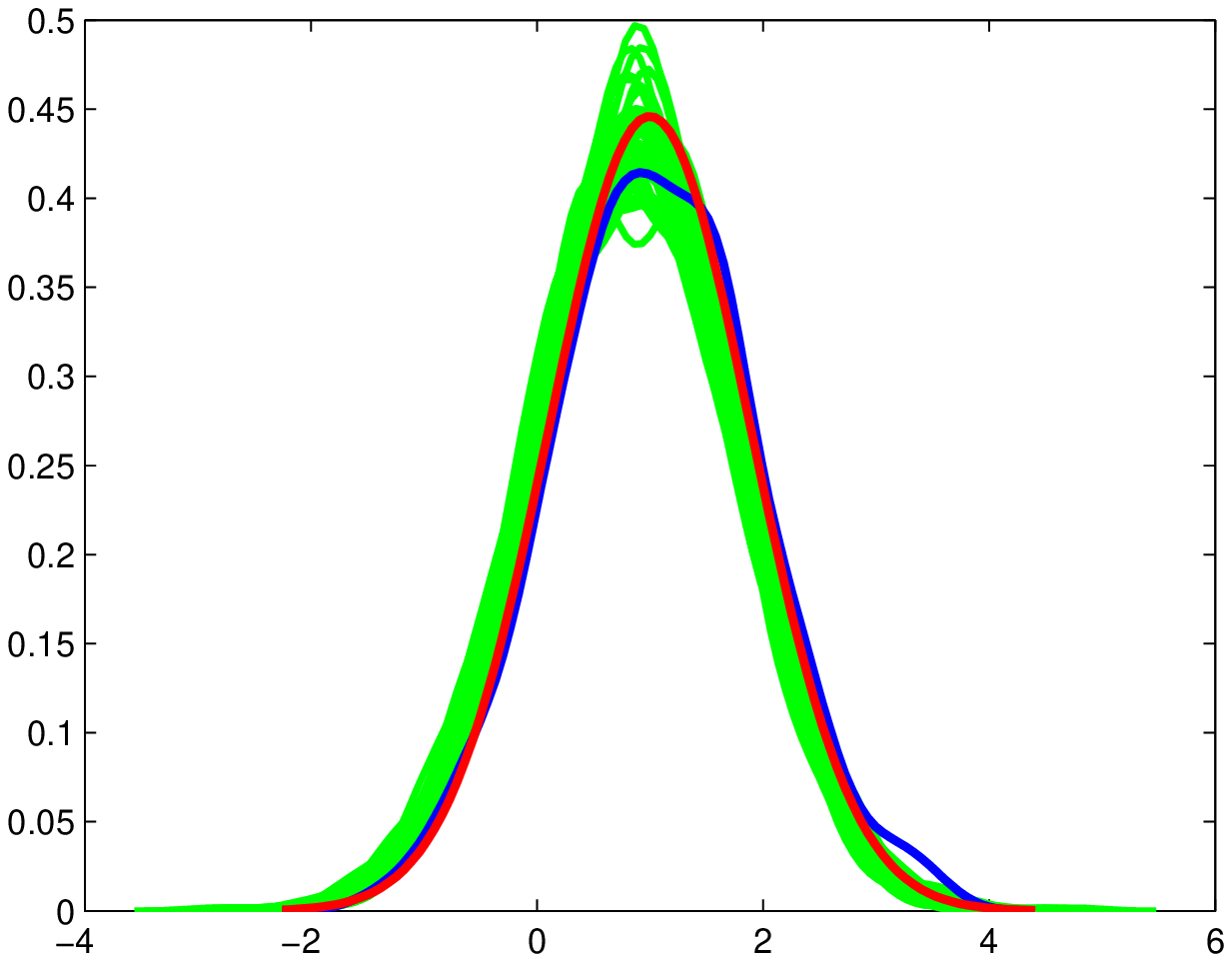}\\
  \rotatebox{90}{\parbox[c]{67mm}{\hspace*{10mm} \bf Gamma random effects}}
  &  \includegraphics[scale=0.95,totalheight=7cm,width=0.36\textwidth]{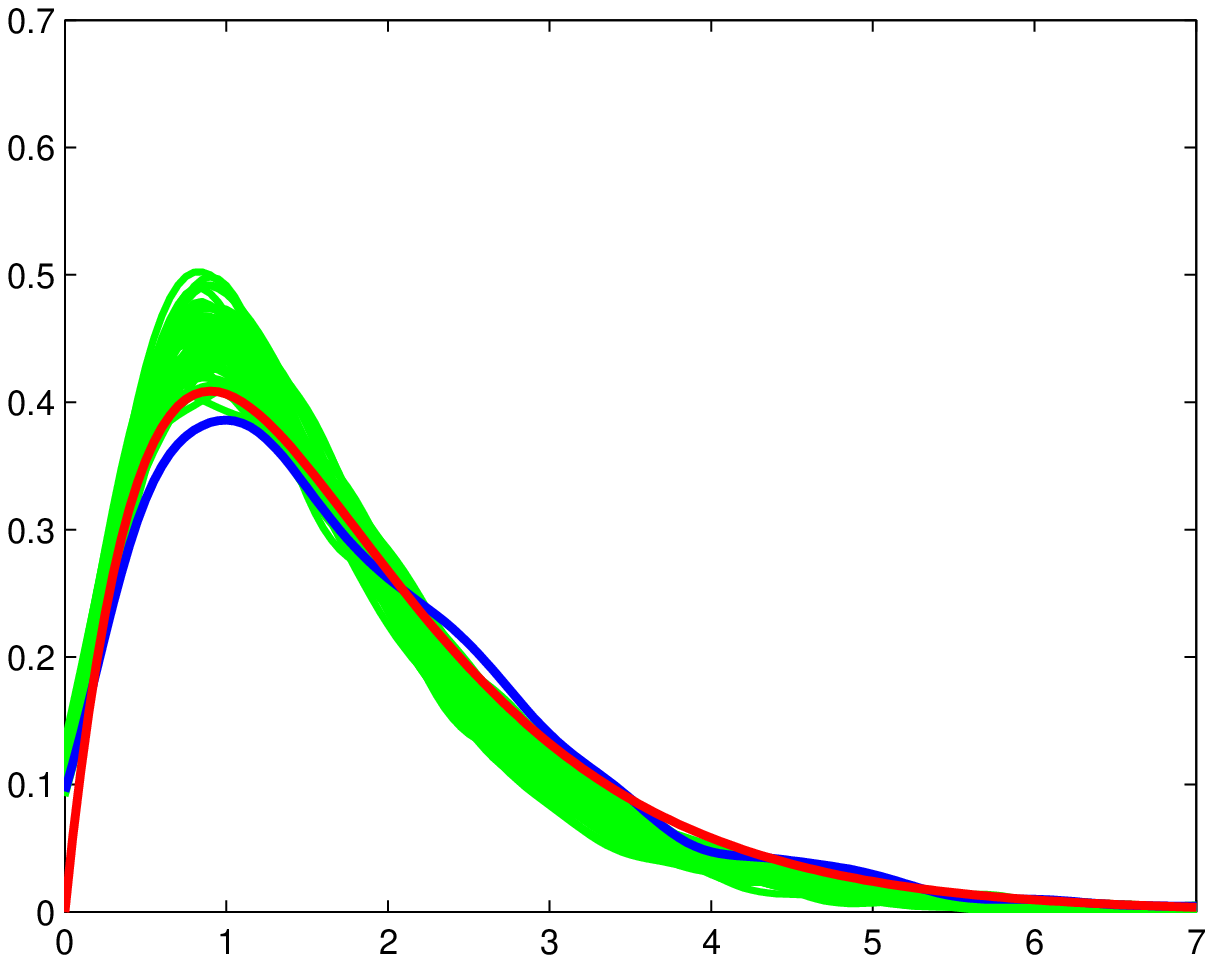}
  & \hspace{-0.7cm}\includegraphics[scale=0.95,totalheight=7cm,width=0.36\textwidth ]{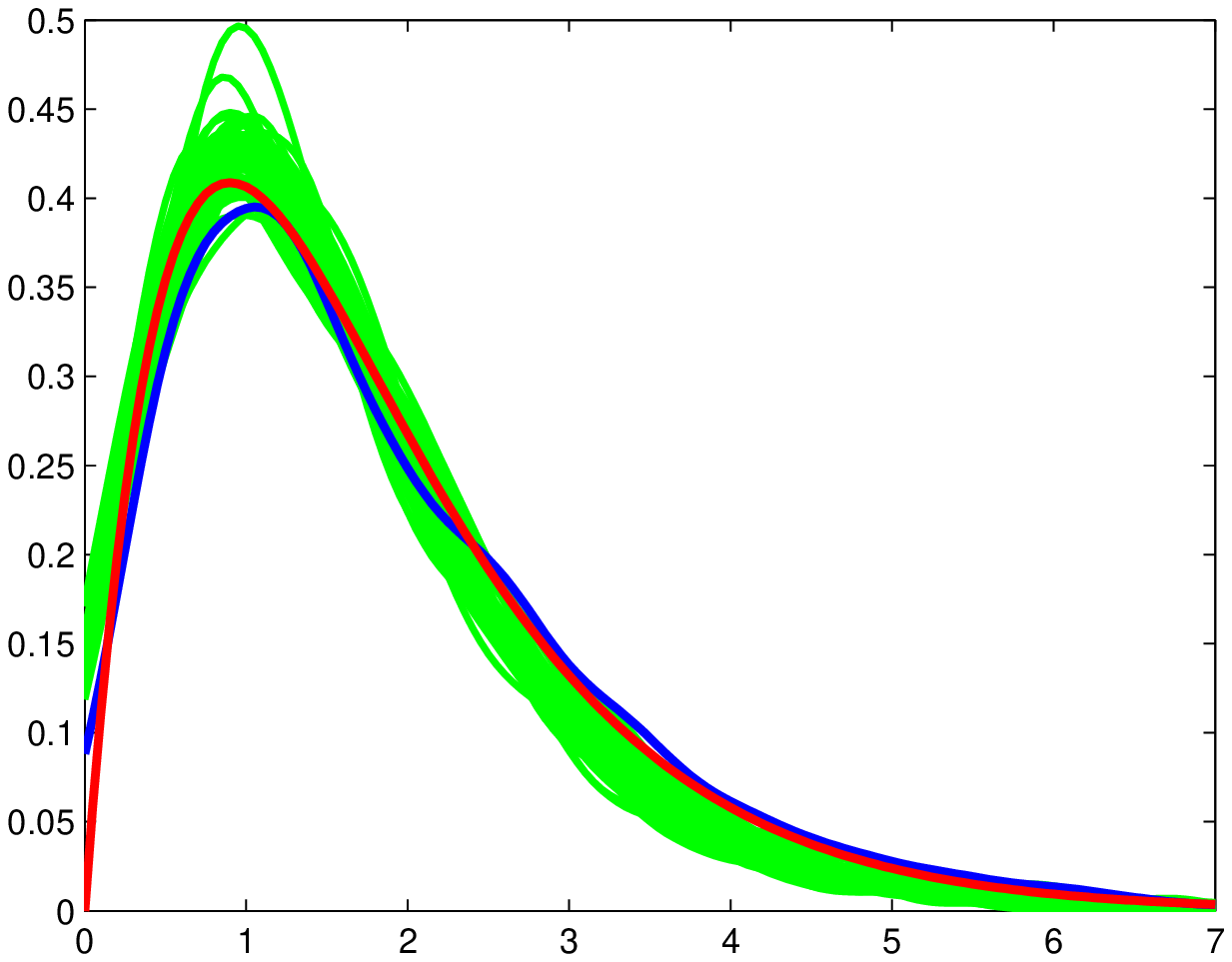}
  & \hspace{-0.7cm}\includegraphics[scale=0.95,totalheight=7cm,width=0.36\textwidth ]{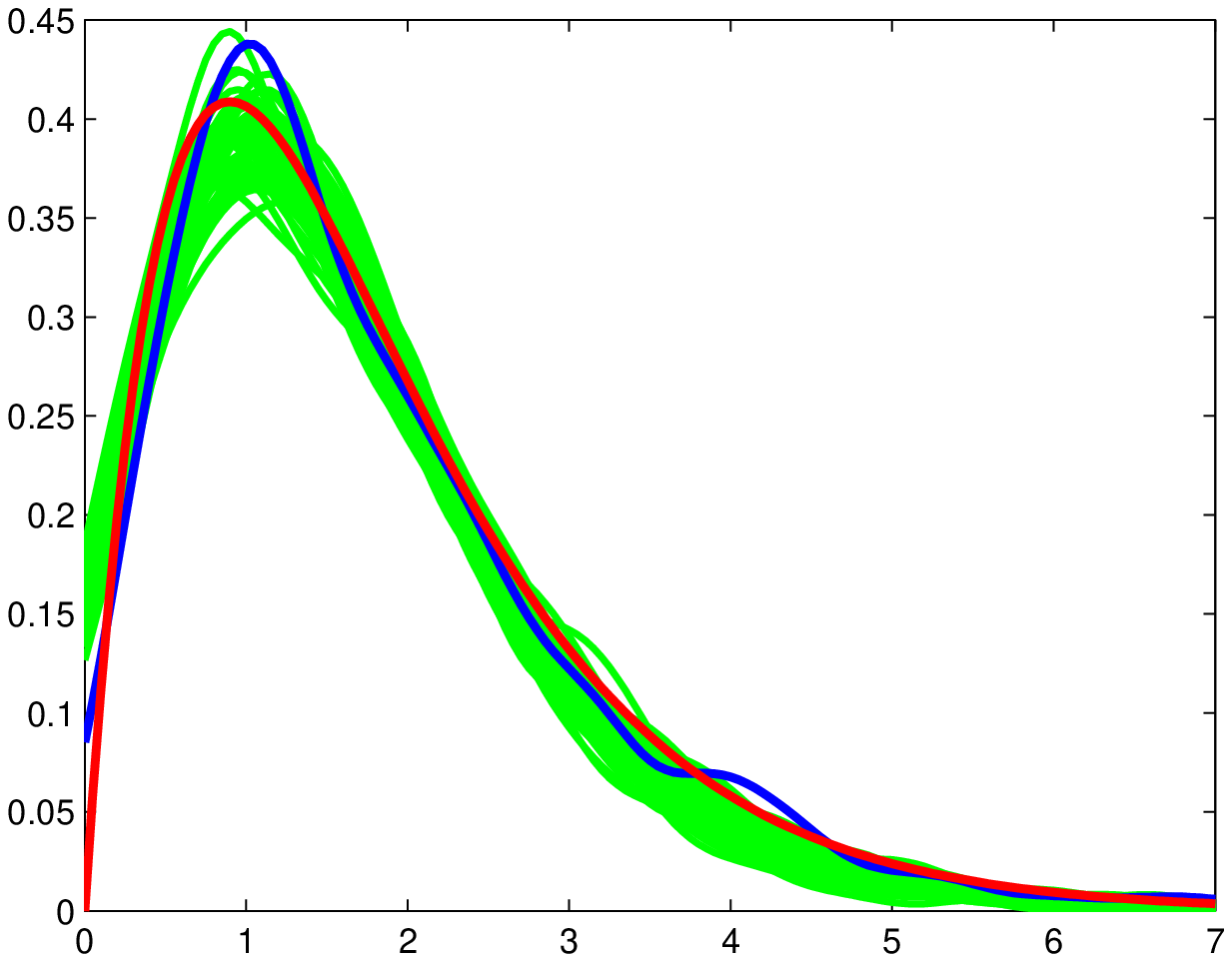}
   \end{tabular}

\caption{{\color{blue}{\bf Kernel estimates $\hat{f}^{(1)}_h$ for Ornstein-Uhlenbeck process with additive random effects:}
We drew 50 i.i.d. realizations of model (\ref{EXX}) for each of the following settings. \emph{First row:} Gaussian distributed random effects, \emph{second row:} gamma distributed random effects, \emph{columns:} different values for the Hurst index~$H$. The thin green lines show the 25 estimated kernel estimates~$\hat{f}^{(1)}_h$. The true density is shown in bold red, and a standard kernel density estimator for one sample of $\phi_i$'s (which is unobserved in a real-case scenario) in blue bold. We chose $N = 1000$ and $T = 100$. For more details, see Section~\ref{Sec.3.7}.}
}
\label{Fig1}
  \end{center}
\end{figure}

\begin{figure}
  \begin{center}
  \begin{tabular}{cccc}
  & {\bf $H=0.25$} & {\bf $H=0.75$} & {\bf $H=0.85$}\\
  \rotatebox{90}{\parbox[c]{67mm}{\hspace*{10mm} \bf Gaussian random effects}}
    &  \includegraphics[scale=0.95,totalheight=7cm,width=0.36\textwidth]{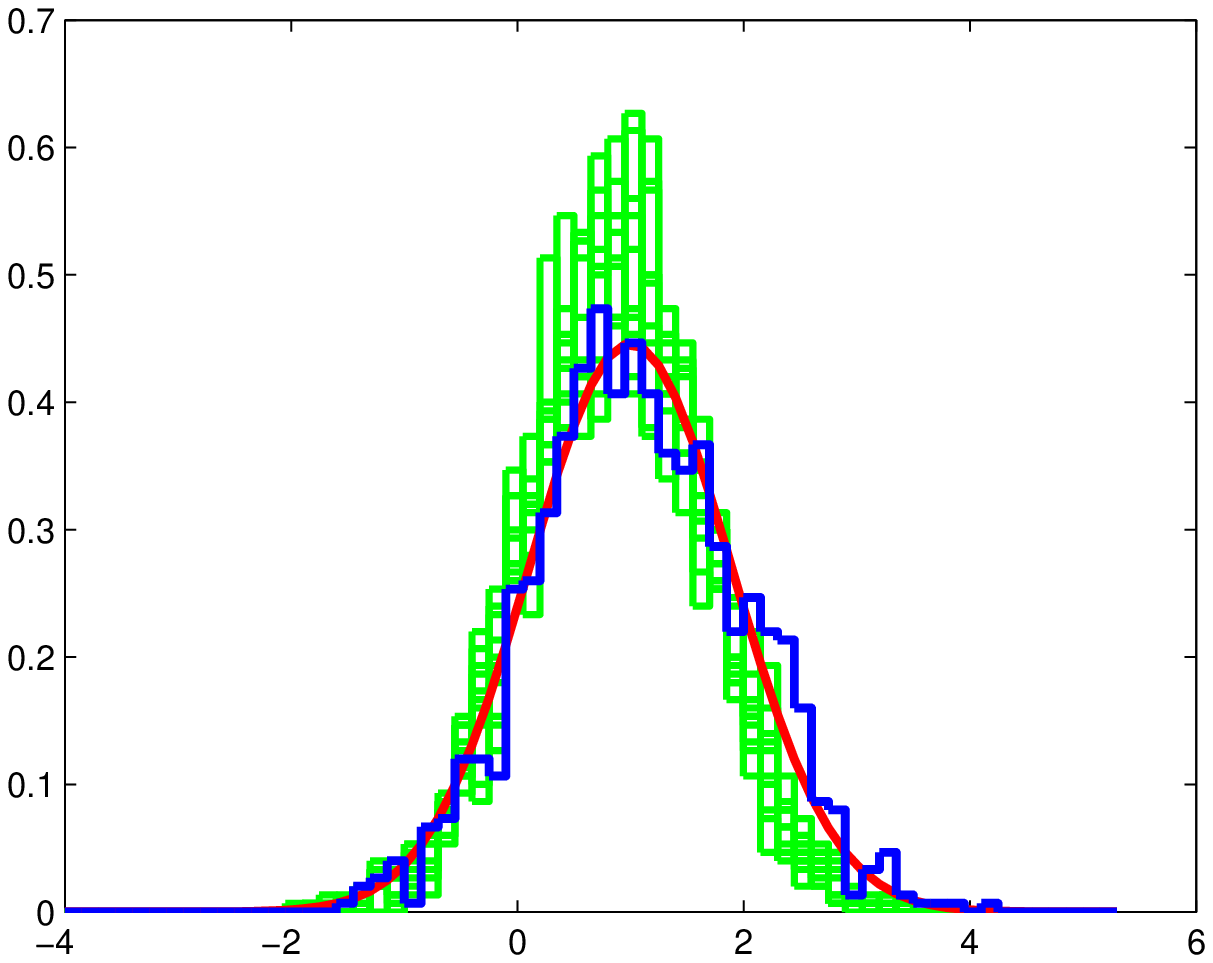}
    & \hspace{-0.7cm} \includegraphics[scale=0.95,totalheight=7cm,width=0.36\textwidth]{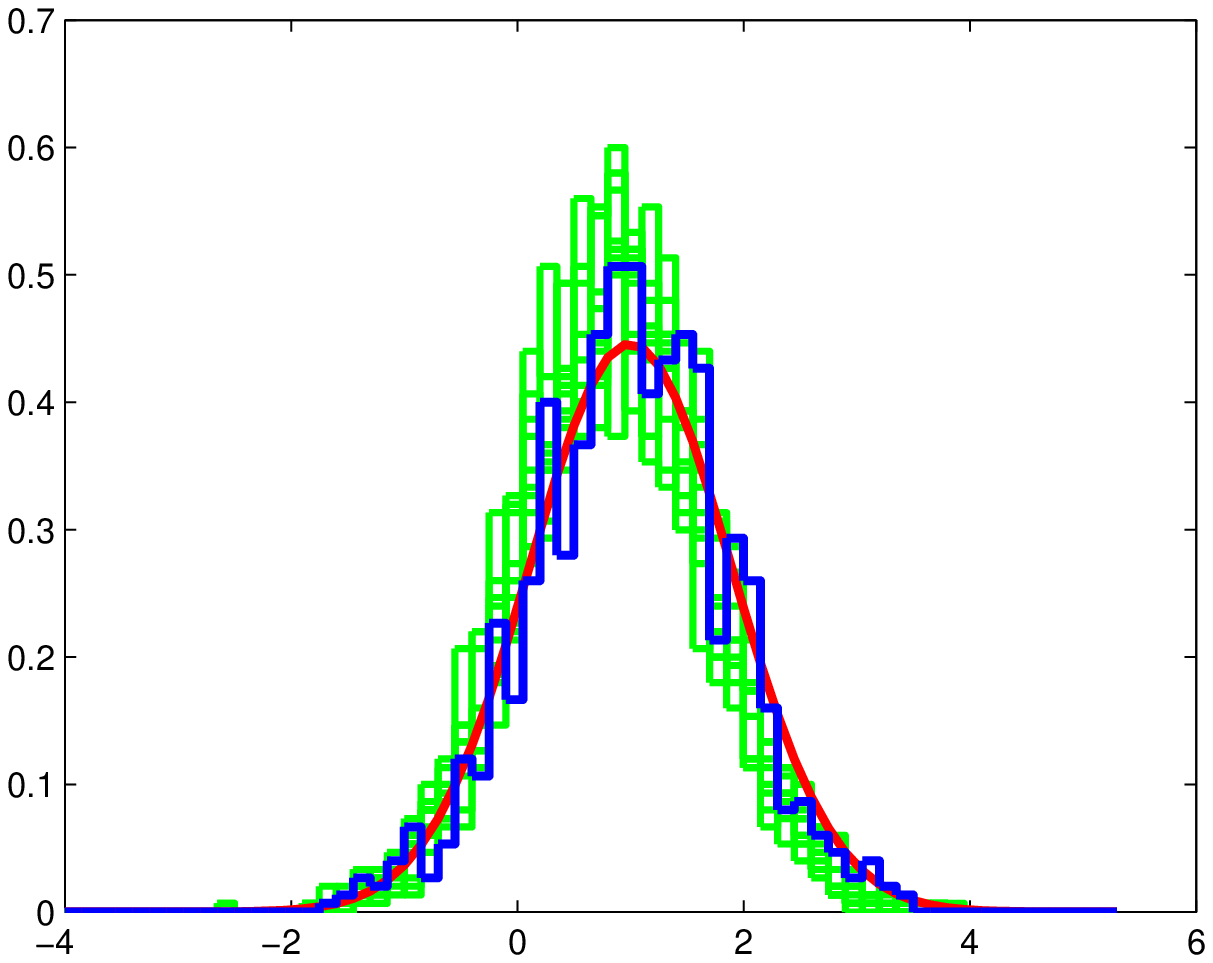}
    & \hspace{-0.7cm} \includegraphics[scale=0.95,totalheight=7cm,width=0.36\textwidth]{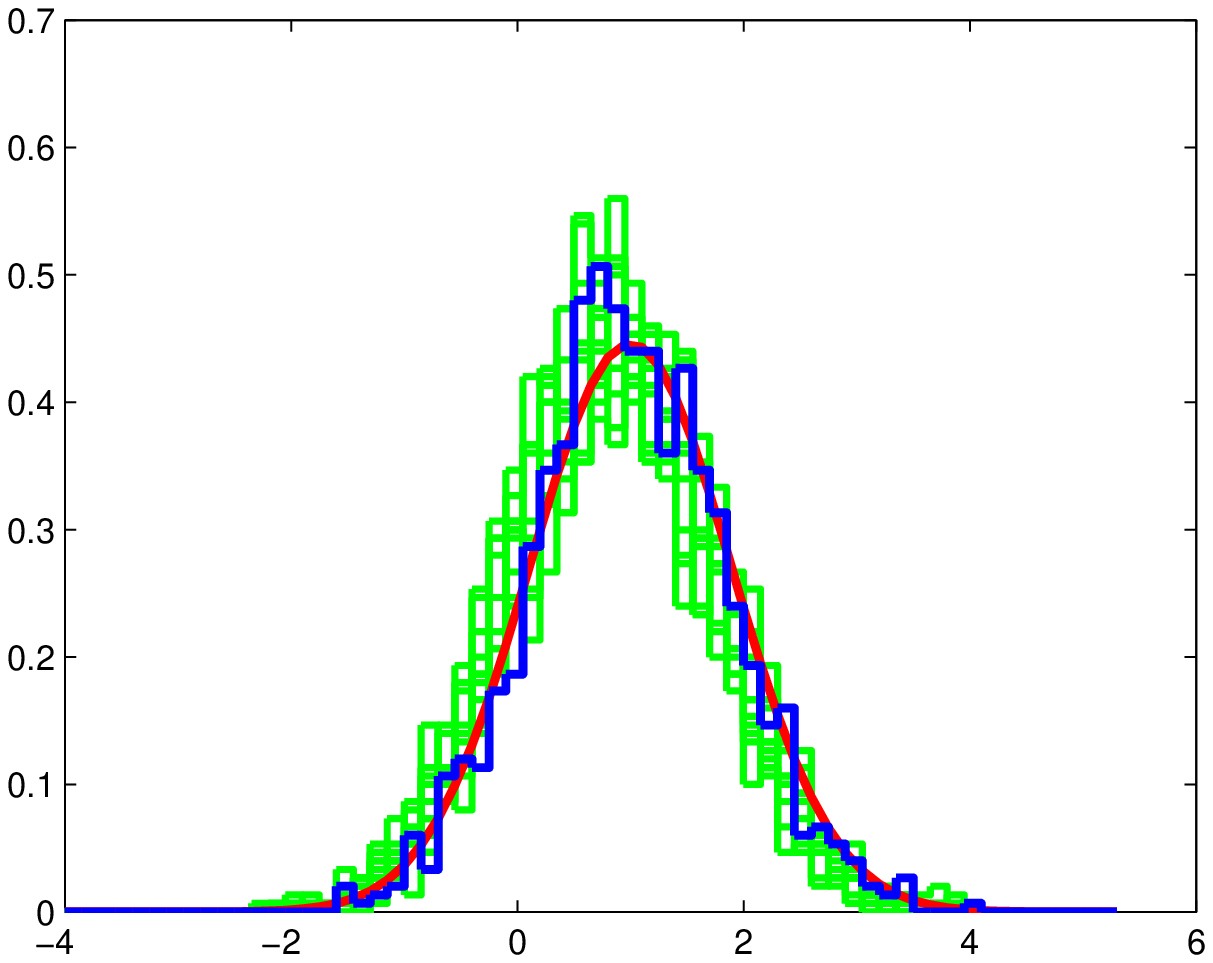}\\
  \rotatebox{90}{\parbox[c]{67mm}{\hspace*{10mm} \bf Gamma random effects}}
  &  \includegraphics[scale=0.95,totalheight=7cm,width=0.36\textwidth]{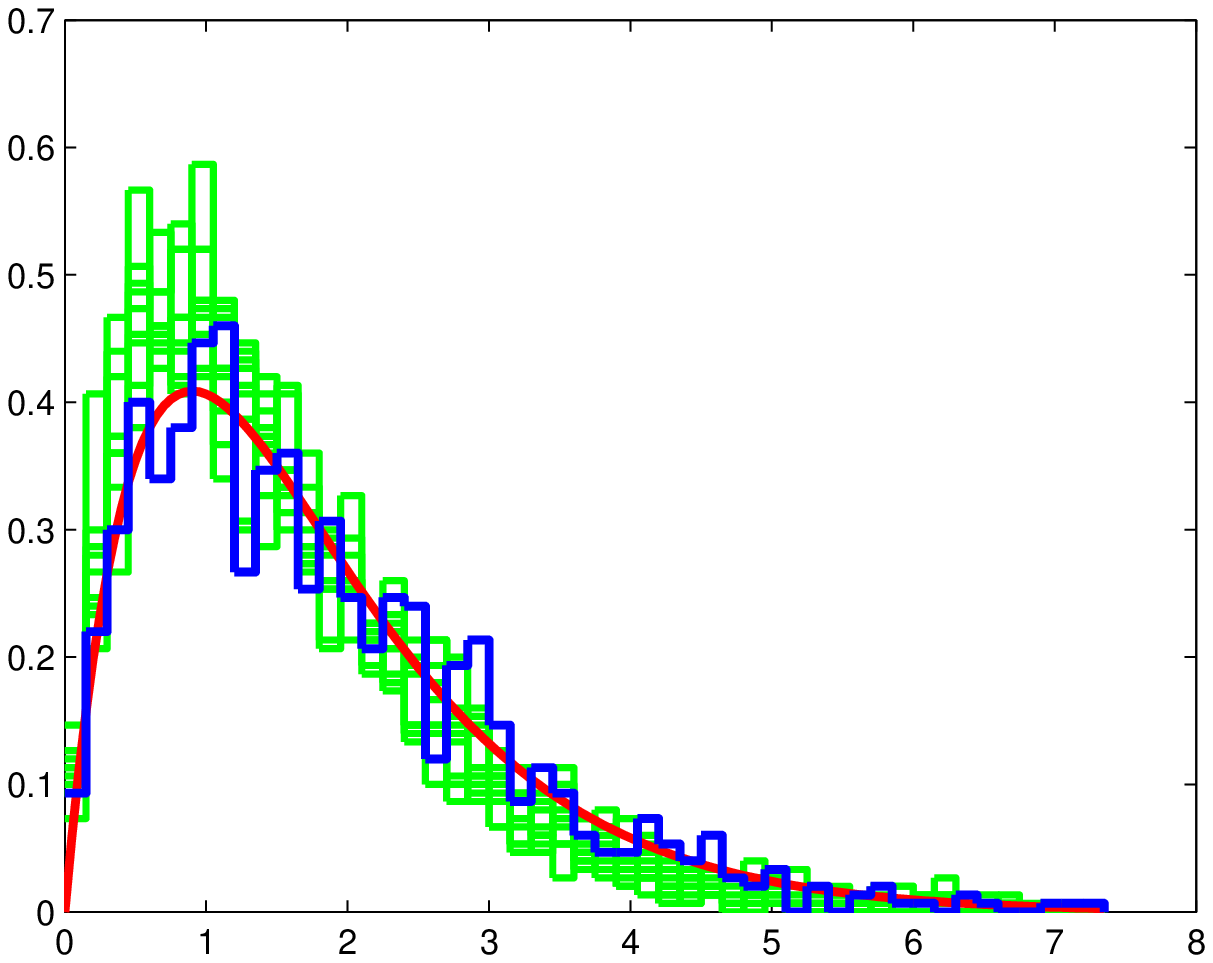}
  & \hspace{-0.7cm}\includegraphics[scale=0.95,totalheight=7cm,width=0.36\textwidth ]{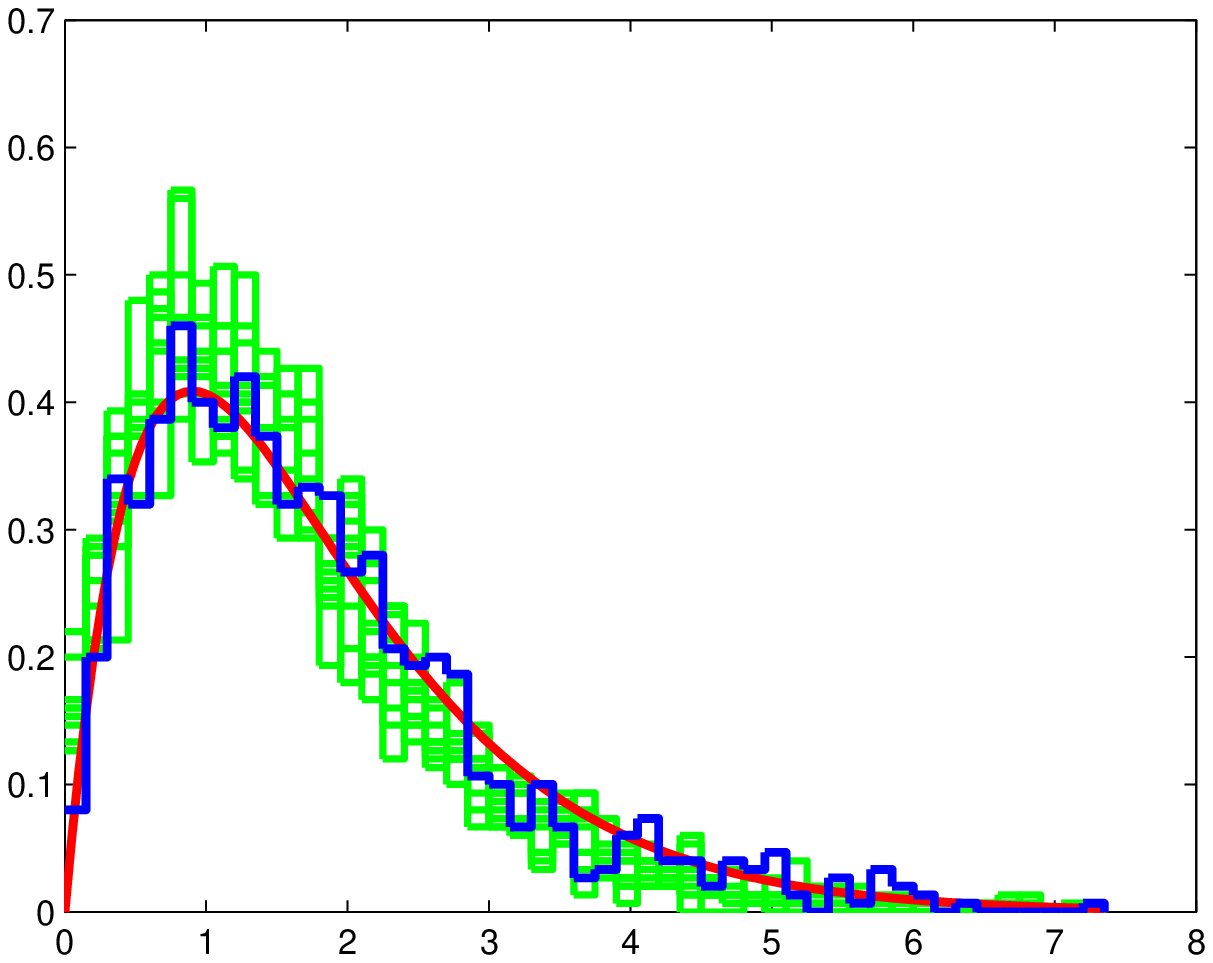}
  & \hspace{-0.7cm}\includegraphics[scale=0.95,totalheight=7cm,width=0.36\textwidth ]{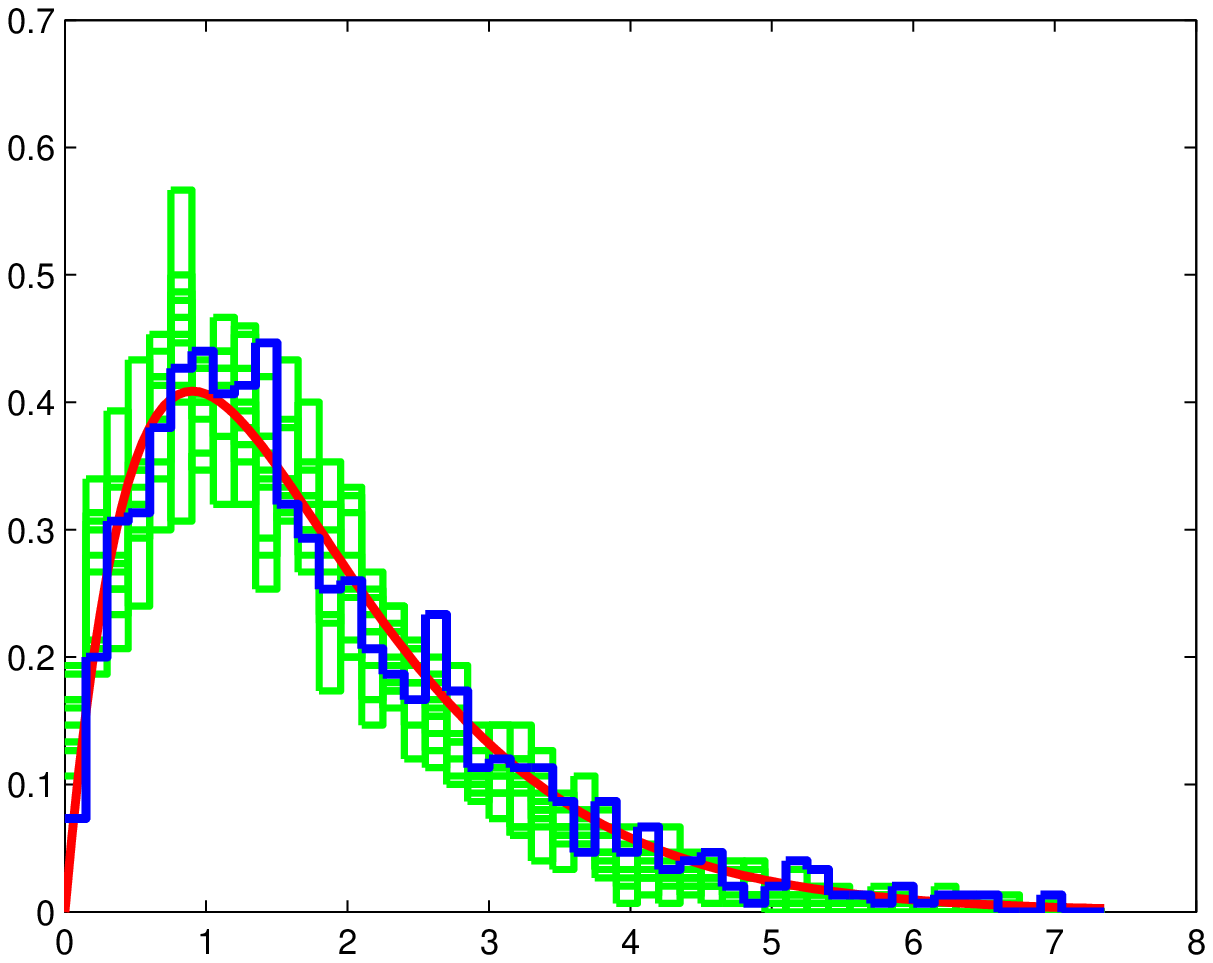}
   \end{tabular}

\caption{{\color{blue}{\bf Histogram estimates $\hat{f}^{(3)}_h$ for Ornstein-Uhlenbeck process with additive random effects:}
We drew 10 i.i.d. realizations of model (\ref{EXX}) for each of the following settings. \emph{First row:} Gaussian distributed random effects, \emph{second row:} gamma distributed random effects, \emph{columns:} different values for the Hurst index~$H$. The thin green lines show the 10 estimated histogram estimates~$\hat{f}^{(3)}_h$. The true density is shown in bold red, and an exact histogram for one sample of $\phi_i$'s (which is unobserved in a real-case scenario) in blue bold. We chose $N = 1000$ and $T = 100$. For more details, see Section~\ref{Sec.3.7}.}
}
\label{Fig2}
  \end{center}
\end{figure}

%%%%%%%%%%%%%%%%%%%%%%%%%%%% T=10  %%%%%%%%%%%%%%%%%%%%%%%%%%%%%%%%%%%
\begin{figure}
  \begin{center}
  \begin{tabular}{cccc}
  & {\bf $H=0.25$} & {\bf $H=0.75$} & {\bf $H=0.85$}\\
  \rotatebox{90}{\parbox[c]{67mm}{\hspace*{10mm} \bf Gaussian random effects}}
    &  \includegraphics[scale=0.95,totalheight=7cm,width=0.36\textwidth]{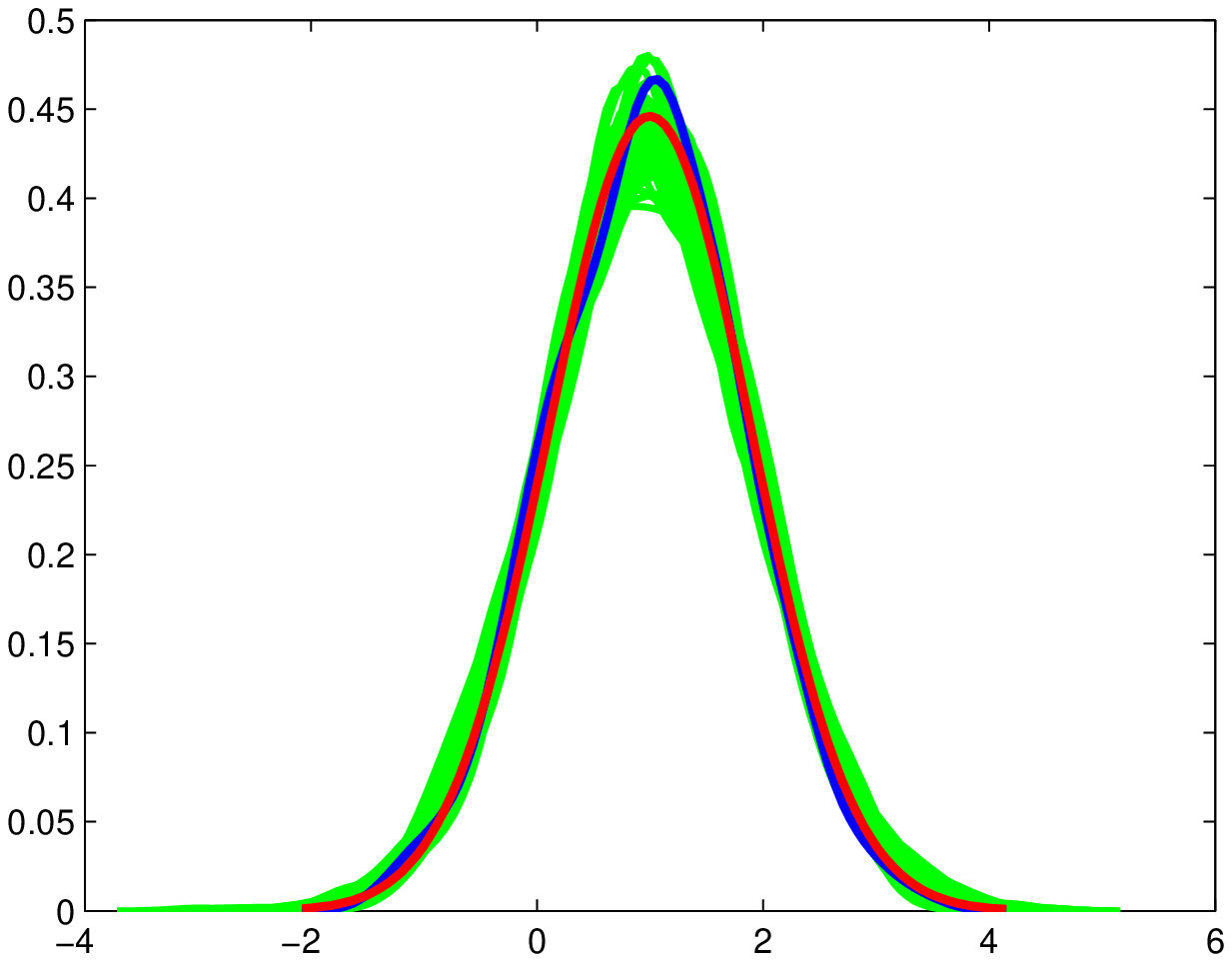}
   &  \hspace{-0.7cm}\includegraphics[scale=0.95,totalheight=7cm,width=0.36\textwidth]{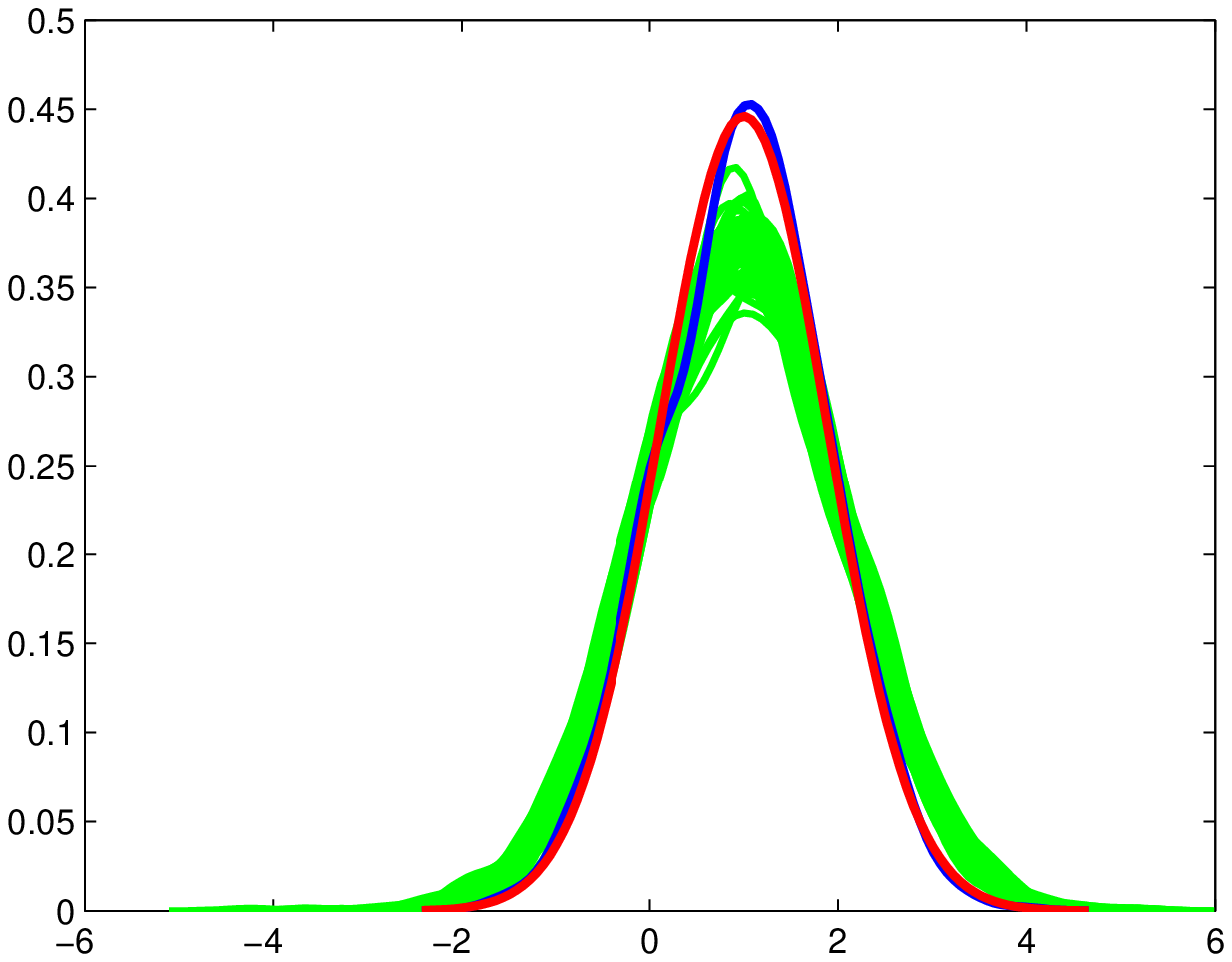}
    & \hspace{-0.7cm} \includegraphics[scale=0.95,totalheight=7cm,width=0.36\textwidth]{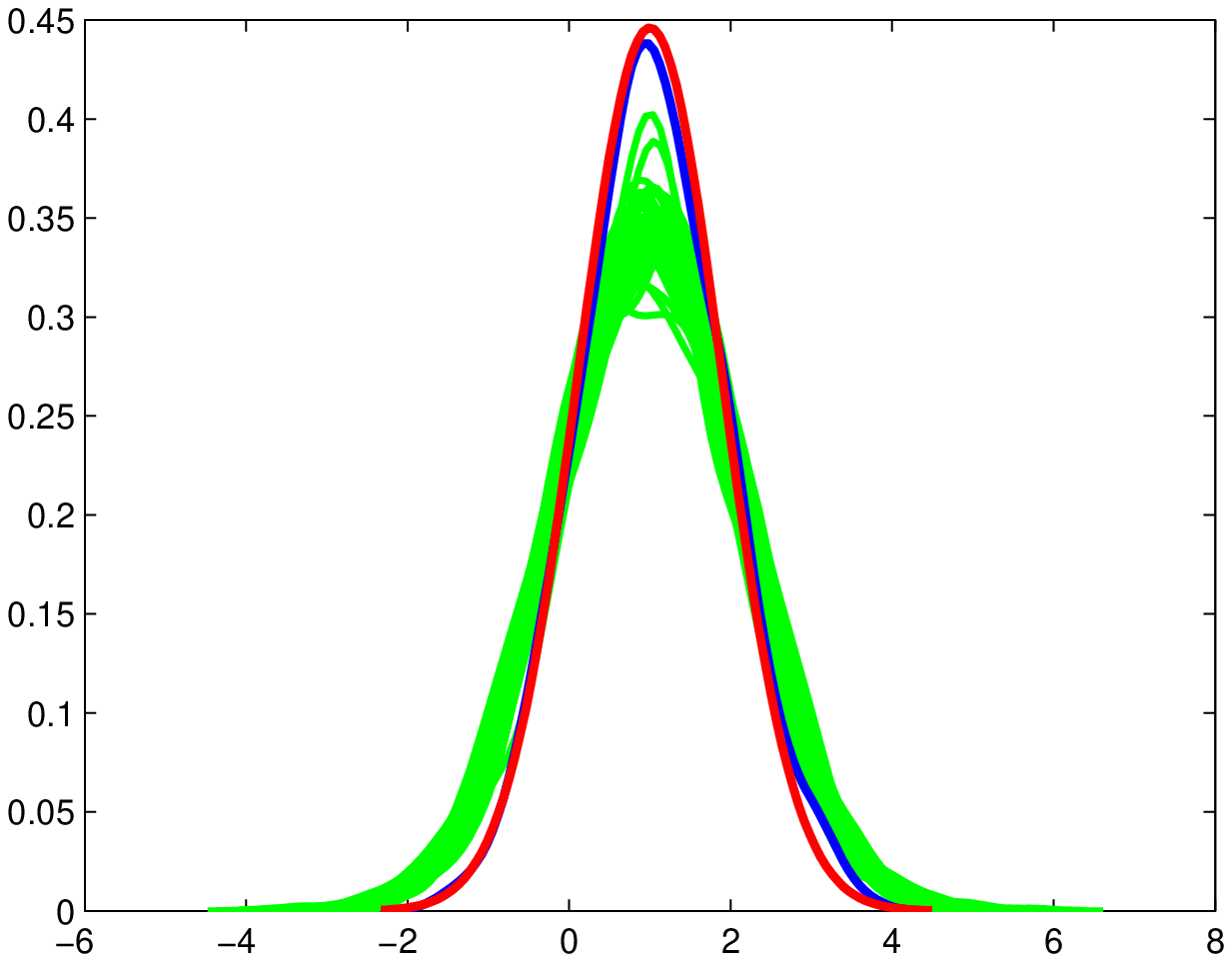}\\
  \rotatebox{90}{\parbox[c]{67mm}{\hspace*{10mm} \bf Gamma random effects}}
  &  \includegraphics[scale=0.95,totalheight=7cm,width=0.36\textwidth]{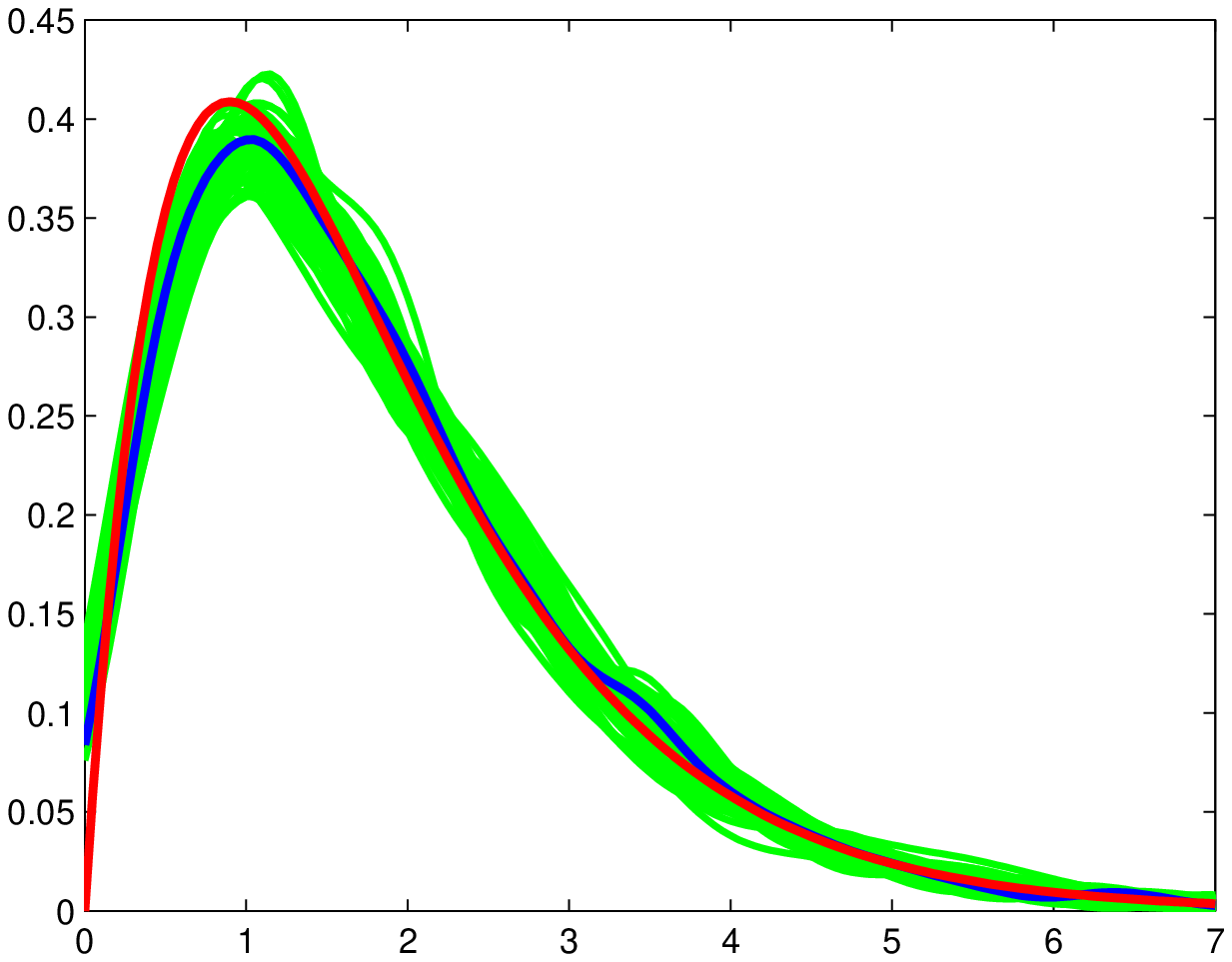}
  & \hspace{-0.7cm}\includegraphics[scale=0.95,totalheight=7cm,width=0.36\textwidth ]{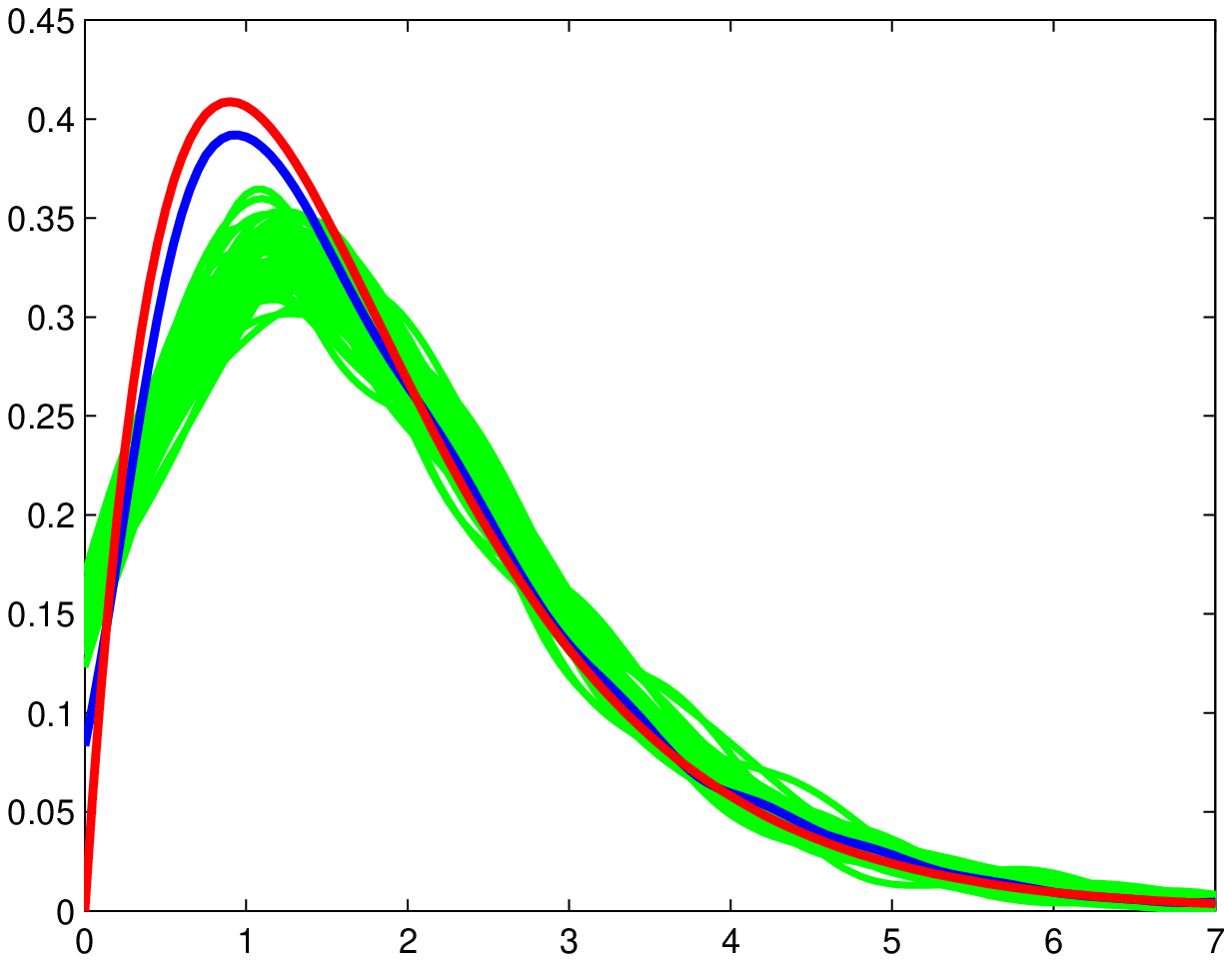}
  & \hspace{-0.7cm}\includegraphics[scale=0.95,totalheight=7cm,width=0.36\textwidth ]{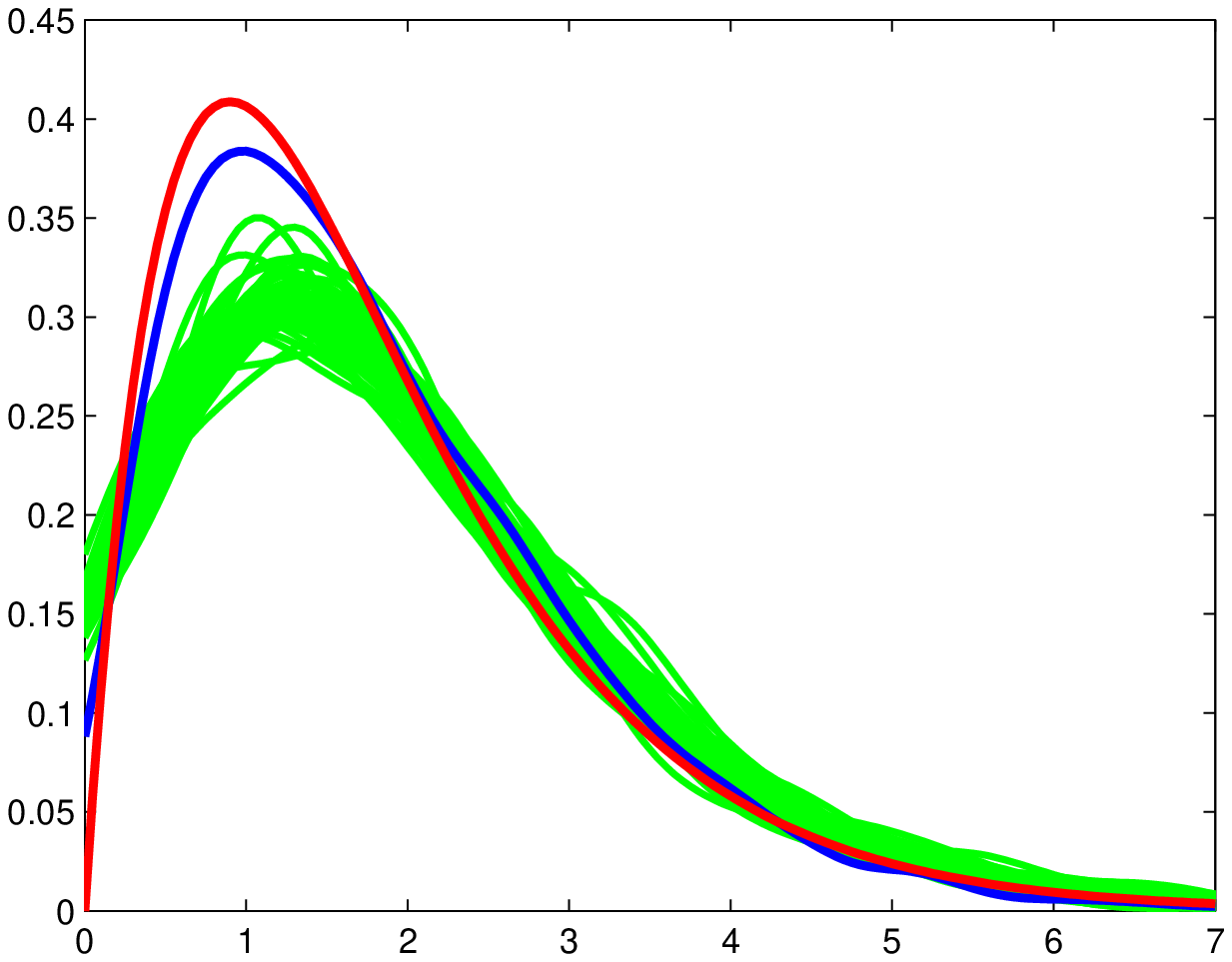}
   \end{tabular}

\caption{{\color{blue}{\bf Kernel estimates $\hat{f}^{(1)}_h$ for Ornstein-Uhlenbeck process with additive random effects:}
We drew 50 i.i.d. realizations of model (\ref{EXX}) for each of the following settings. \emph{First row:} Gaussian distributed random effects, \emph{second row:} gamma distributed random effects, \emph{columns:} different values for the Hurst index~$H$. The thin green lines show the 25 estimated kernel estimates~$\hat{f}^{(1)}_h$. The true density is shown in bold red, and a standard kernel density estimator for one sample of $\phi_i$'s (which is unobserved in a real-case scenario) in blue bold. We chose $N = 1000$ and $T = 10$. For more details, see Section~\ref{Sec.3.7}.}
}
\label{Fig3}
  \end{center}
\end{figure}

\begin{figure}
  \begin{center}
  \begin{tabular}{cccc}
  & {\bf $H=0.25$} & {\bf $H=0.75$} & {\bf $H=0.85$}\\
  \rotatebox{90}{\parbox[c]{67mm}{\hspace*{10mm} \bf Gaussian random effects}}
    &  \includegraphics[scale=0.95,totalheight=7cm,width=0.36\textwidth]{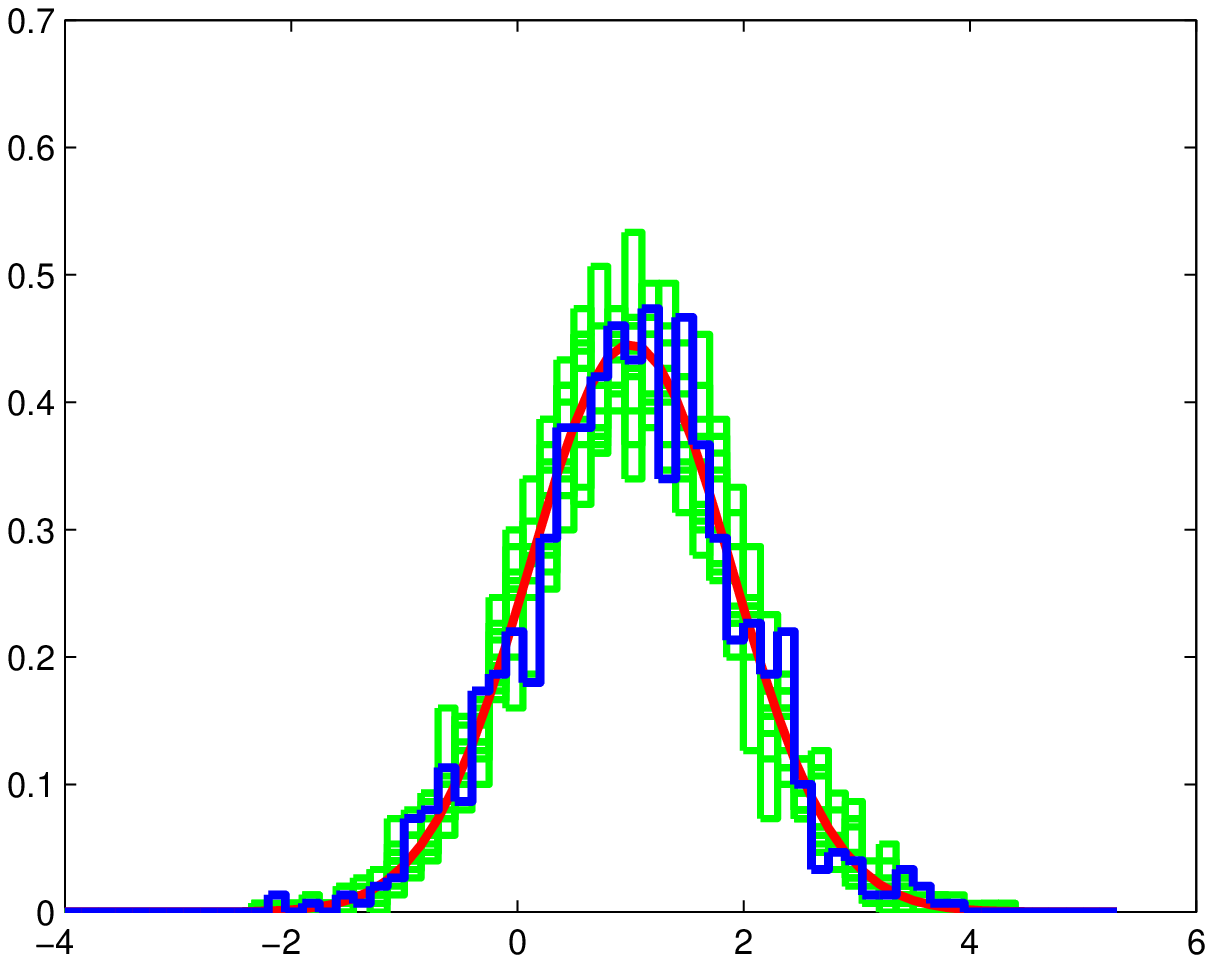}
    & \hspace{-0.7cm} \includegraphics[scale=0.95,totalheight=7cm,width=0.36\textwidth]{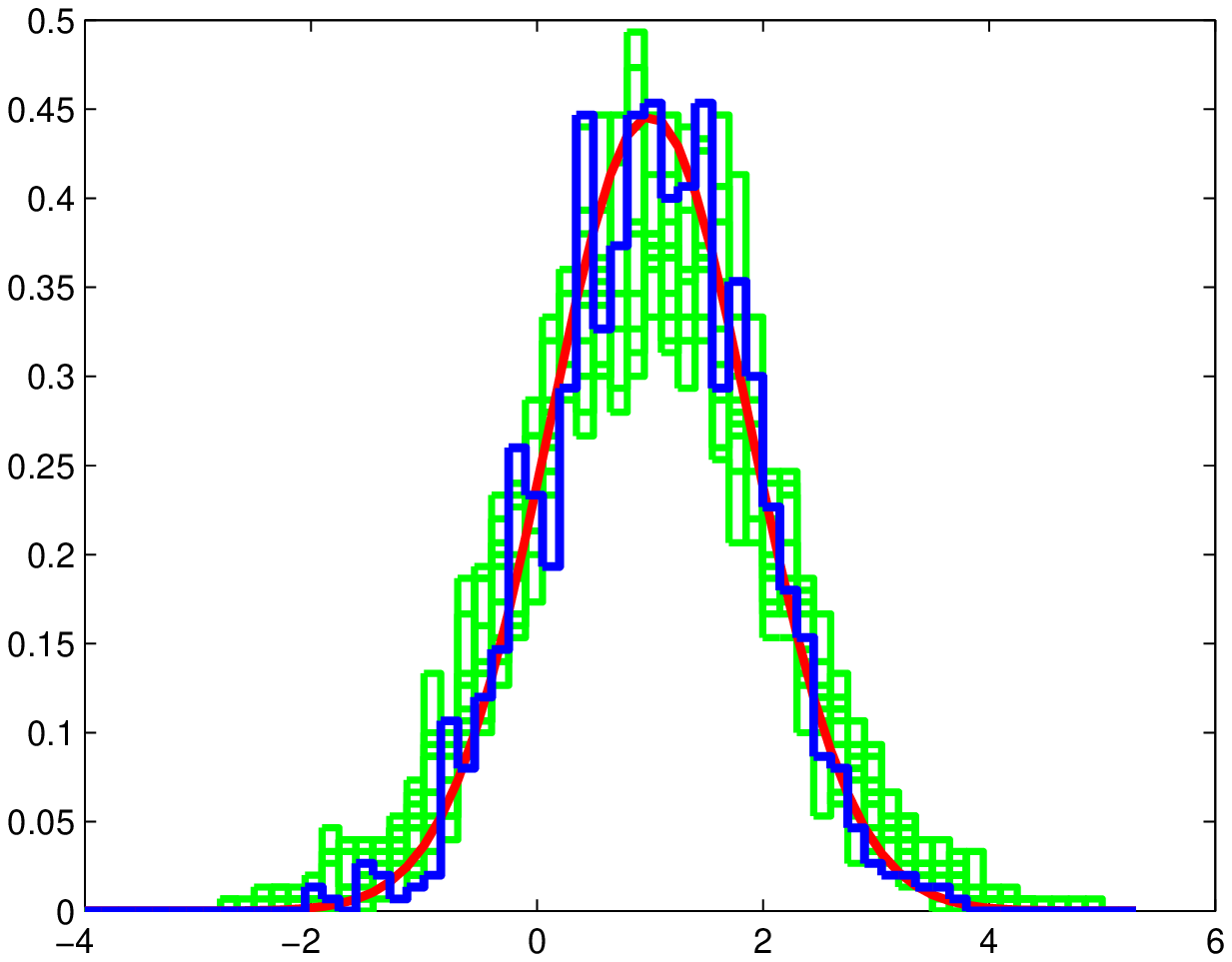}
    & \hspace{-0.7cm} \includegraphics[scale=0.95,totalheight=7cm,width=0.36\textwidth]{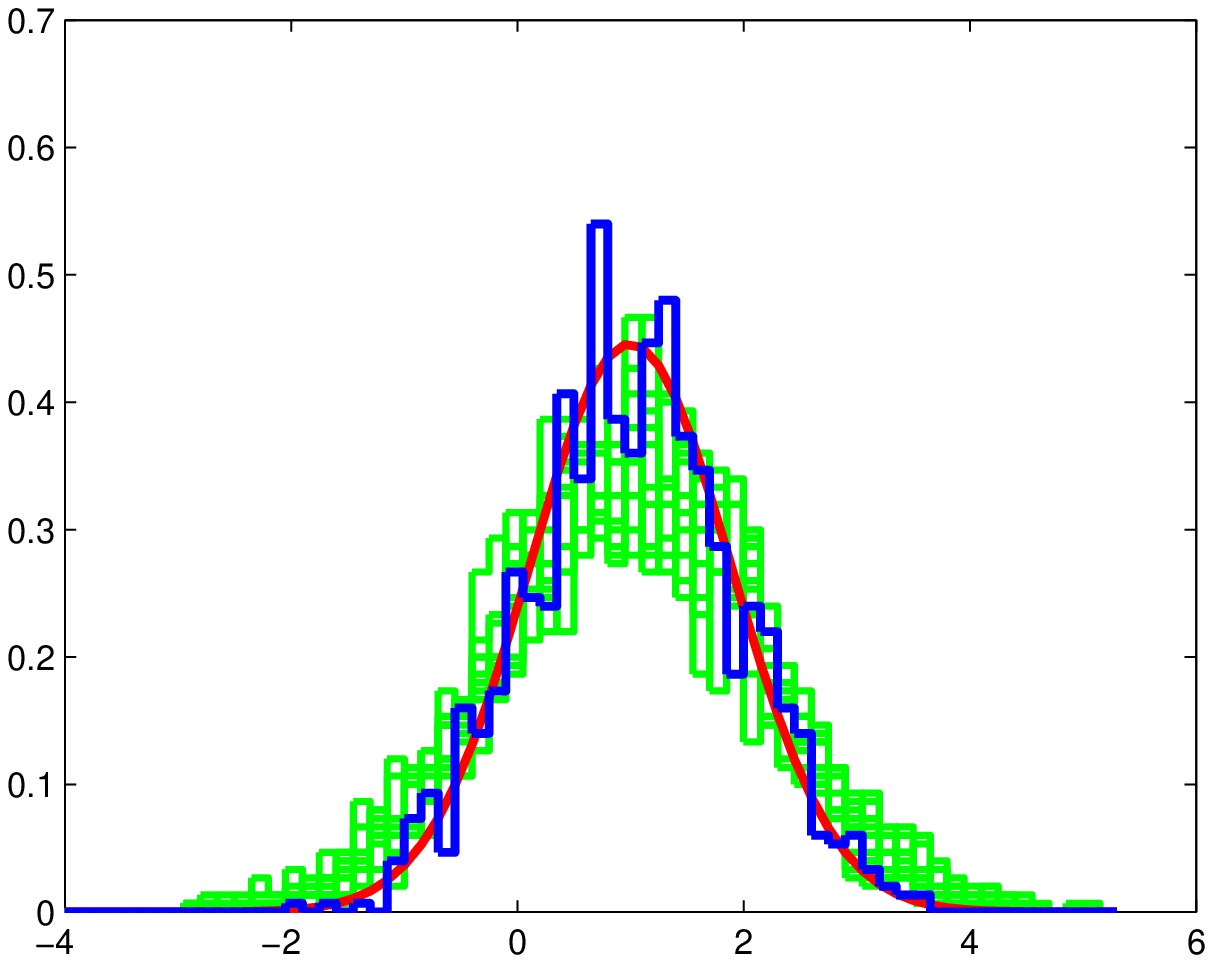}\\
  \rotatebox{90}{\parbox[c]{67mm}{\hspace*{10mm} \bf Gamma random effects}}
  &  \includegraphics[scale=0.95,totalheight=7cm,width=0.36\textwidth]{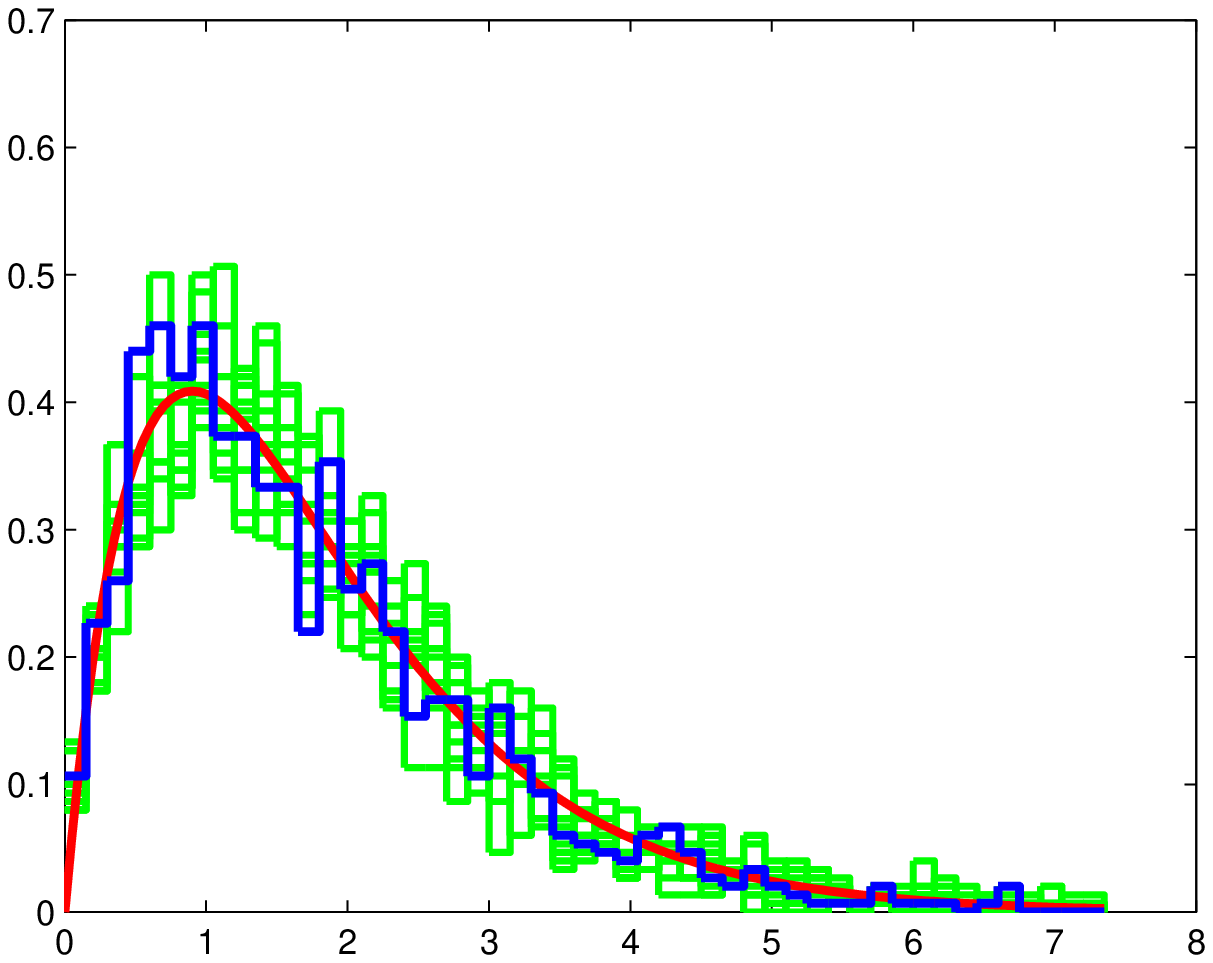}
  & \hspace{-0.7cm}\includegraphics[scale=0.95,totalheight=7cm,width=0.36\textwidth ]{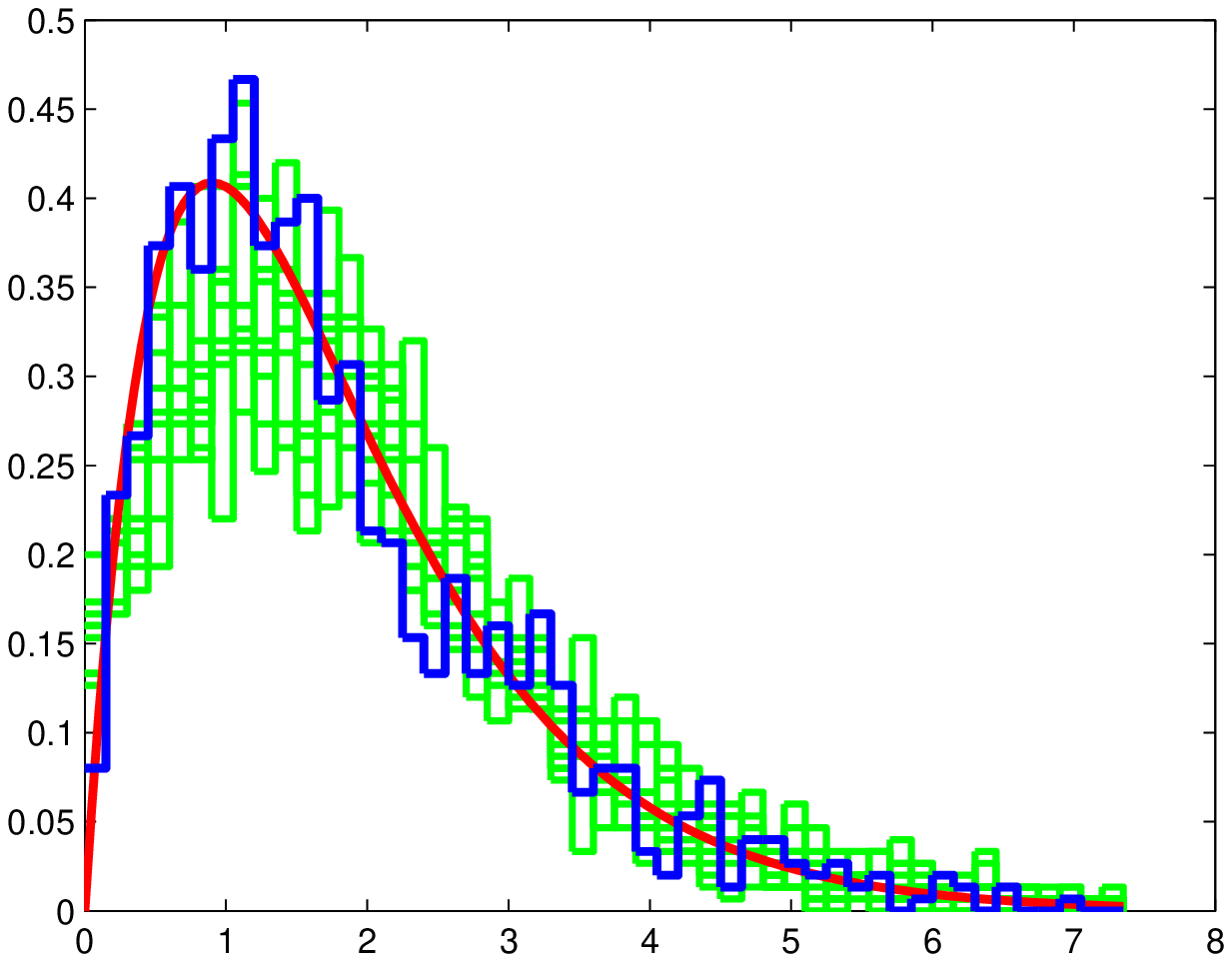}
  & \hspace{-0.7cm}\includegraphics[scale=0.95,totalheight=7cm,width=0.36\textwidth ]{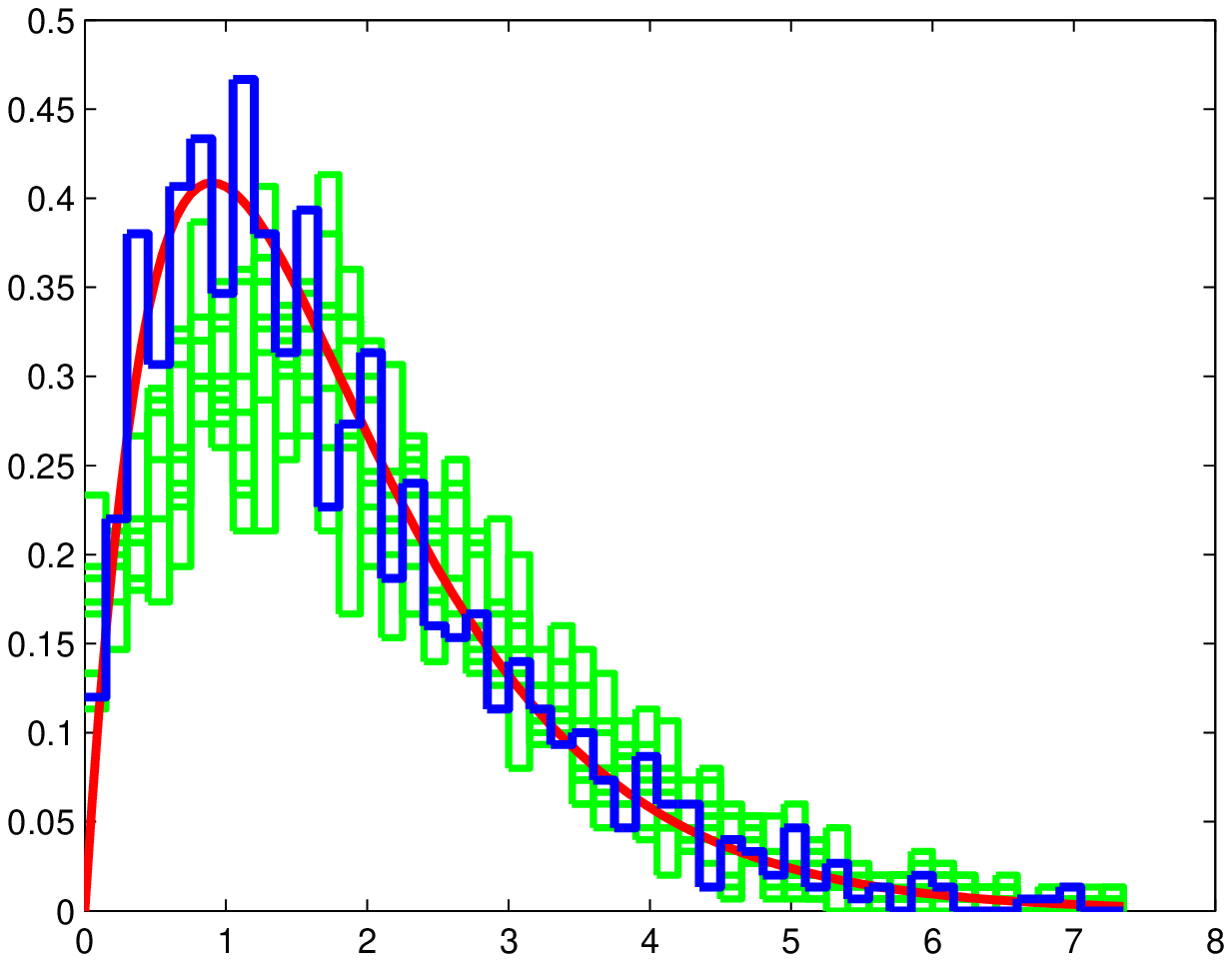}
   \end{tabular}

\caption{{\color{blue}{\bf Histogram estimates $\hat{f}^{(3)}_h$ for Ornstein-Uhlenbeck process with additive random effects:}
We drew 10 i.i.d. realizations of model (\ref{EXX}) for each of the following settings. \emph{First row:} Gaussian distributed random effects, \emph{second row:} gamma distributed random effects, \emph{columns:} different values for the Hurst index~$H$. The thin green lines show the 10 estimated histogram estimates~$\hat{f}^{(3)}_h$. The true density is shown in bold red, and an exact histogram for one sample of $\phi_i$'s (which is unobserved in a real-case scenario) in blue bold. We chose $N = 1000$ and $T = 10$. For more details, see Section~\ref{Sec.3.7}.}
}
\label{Fig4}
  \end{center}
\end{figure}

%%%%%%%%%%%%%%%%%%%%%%%%%%%%%%%%%%%%%%%%%%%%%%%%%%%%%%%%%%%%%%%%%%%%%%%%%%%%%%%%%%%%%%%%%%%%%%%%%%%%%%%%%%%%%%%%%

%\todo[inline,color=blue!20]{We're replacing the plots with some of the correct width/height ratio such that they don't have to be scaled; should be much better readable then.}

 \section{Concluding remarks}\label{Sec.4} %{\color{blue}
 To summarize, we addressed the open research question of how to estimate the density of random effects in fractional stochastic diferential equations in a nonparametric fashion. To that end, we considered $N$ i.i.d processes $\displaystyle X^i(t),~0\leq t\leq T, i=1,\cdots,N$, where the dynamics of $X^i$ was described by an FSDE including a random effect $\phi_i$. The nonparametric estimation of the density of $\phi_i$ was investigated for a general linear model of the form $\displaystyle dX_t=\left(a(X_t)+\phi b(t) \right)dt+\sigma(t)dW^H_t$, where $b(\cdot)$ and $\sigma(\cdot)$ were known functions, but $a(\cdot)$ was possibly unknown. We studied the asymptotic behavior of the proposed density estimators for the whole range $\displaystyle H\in\left(1/2,1 \right)$, built kernel density estimators and studied their $L^2$-risk as both $N$ and $T$ tended to infinity. We also provided histogram estimators in a specific case where $f$ had compact support, which was for two reasons: First, we aimed to simplify technical computations, and cases where the random effects density $f$ had unbounded support were less important, since data could always be mapped monotonically to $[0,1]$. Second, densities with unbounded support occur less often in practice. For the proposed histogram estimators, we provided their $L^1$-risk for both $N$ and $T=T(N)$ tending to infinity.\\

%{\color{blue}
Several interesting extensions of the present study are possible: A first direction would be to provide density estimators for short-range dependent models, that is $H<1/2$. For such models  one has to change the assumption \textbf{A}$_1$ since it provides
$$\displaystyle\liminf_{T\rightarrow\infty} \mathbb{E}\left(\frac{\int_0^T (b(t)/\sigma(t)) dW^H_t }{\int_0^T (b^2(t)/\sigma^2(t)) dt }\right)^2\geq B_H>0,$$
which in turn implies the non-consistency of the proposed estimators $\widetilde{\phi}_{i,T}$ in $L^2$-sense. Another direction would be to consider models with nonlinear drift. In this case, one has to face the problem of estimating random effects $\phi_i$. Methods of parametric estimation, such as the maximum likelihood technique, may help to estimate these random effects.

%\section*{References}
\bibliographystyle{elsarticle-num}

\begin{thebibliography}{00}

 \bibitem{Antic et al. 2009} Antic J, Laffont CM, Chafai D, Condorcet D.  Comparison of nonparametric methods in nonlinear mixed effects models. Comput. Statist. Data Anal. 2009;53: p. 642-–656.

\bibitem{Beran 1994} Beran J.  Statistics for long memory processes. London: Chapman and Hall; 1994.

\bibitem{Bishwal 2008}
Bishwal J. P.N. Parameter estimation in stochastic differential equations. Berlin Heidelberg: Springer-Verlag;2008.

\bibitem{Cheridito et al 2003}
Cheridito P, Kawaguchi H, Maejima M. Fractional Ornstein-Uhlenbeck processes. Electr J Probab.
2003;8, p. 1–14

\bibitem{Coeurjolly 2001} Coeurjolly JF.  Estimating the parameters of a fractional Brownian motion by discrete variations of its sample paths. Stat Inference Stoch Process. 2001; 4: p. 199-–227.

\bibitem{Collins and De Luca 1994} Collins JJ, De Luca CJ. Upright, correlated random walks: A statistical-biomechanics approachto the human postural control system. Chaos 5,1995; 1, p. 57--63.

\bibitem{Comte and Samson 2012} Comte F, Samson A. Nonparametric estimation of random effects densities in linear mixed-effects model. J. Nonparametr. Stat.2012; 24 (4), p. 951-–975.

\bibitem{Delattre et al. 2012} Delattre M, Genon-Catalot V, Samson A.  Maximum likelihood estimation for stochastic differential equations with random effects. Scand. J. Stat. 2012;\textbf{40}, p. 322--343.

\bibitem{Devroye and Gyorfi 1985} Devroye L, Gyorfi L. Nonparametric density estimation: the $L_1$ view,  John Wiley \& Sons, Inc. 1985.

\bibitem{Ditlevsen and  De Gaetano 2005a}
 Ditlevsen S, De Gaetano A. Mixed effects in stochastic differential equation models. REVSTAT - Statist. J.
 2005a;\textbf{3}, p. 137--153.

\bibitem{Granger 1966} Granger CWJ. The typical spectral shape of an economic variable. Econometrica. 1966; 34, p. 150--161.


\bibitem{Hu and Nualart 2010} Hu Y, Nualart D. Parameter estimation for fractional Ornstein-Uhlenbeck processes. Statist. Probab. Lett. 2010; 80(11-–12), p. 1030-–1038.

\bibitem{Hu et al. 2011} Hu Y, Nualart D, Xiao W, et al. Exact maximum likelihood estimator for drift fractional Brownian motion at discrete observation.  Acta Math. Sci. Ser. B Engl. Ed. 2011;31(5), p. 1851-–1859.

\bibitem{Kleptsyna and Le Breton 2002} Kleptsyna ML, Le Breton A. Statistical analysis of the fractional Ornstein-Uhlenbeck type process. Stat. Inference Stoch. Process. 2002;5(3), p. 229–-248.

\bibitem{Kuklinski et al. 1989} Kuklinski WS, Chandra K, Ruttimann UE, Webber RL. Application of Fractal Texture Analysis to Segmentation of Dental Radiographs SPIE, Medical Imaging III: Image Processing. 1989;1092 , p. 111--117.

\bibitem {Kutoyants 1984} Kutoyants YA. Parameter estimation for stochastic processes. Research and Exposition in Mathematics, vol. 6. Berlin: Heldermann Verlag; 1984.

\bibitem {Liptser and Shiryaev 2001} Liptser R, Shiryaev A.  Statistics of random processes II: general theory. New York: Springer-Verlag; 2001.

\bibitem{Mandelbort et al. 1968} Mandelbort B, Van Ness, Wallis J. Fractional Brownian motions, fractional noises and applications. SIAM review 1968;10 , p. 422--437.

\bibitem{McLeod and hipel 1978} Mcleod AI, Hipel KW. Preservation of the Rescaled Adjusted Range: a reassessment of the Hurst exponent Water Ressources. Research. 1978;14, 3 , p. 491--508.

\bibitem{Mishura 2008} Mishura YS. Stochastic calculus for fractional Brownian motion and related processes. Lecture Notes in Mathematics, vol. 1929. Berlin: Springer;  2008.

\bibitem{Nie and Yang 2005} Nie L, Yang M.  Strong consistency of the MLE in nonlinear mixed-effects models with large cluster size. Sankhya: The Indian J. Stat. 2005; \textbf{67}, p. 736--763.

\bibitem{Nie 2006} Nie L.  Strong consistency of the maximum likelihood estimator in generalized linear and nonlinear mixed-effects models. Metrika. 2006;\textbf{63}, p. 123--143.

\bibitem{Nie 2007} Nie L. Convergence rate of the MLE in generalized linear and nonlinear mixed-effects models: theory and applications. J. Statist. Plann. Inference. 2007;\textbf{137}, p. 1787--1804.

\bibitem{Nualart and Rascanu 2002} Nualart D, Rascanu A  Differential equations driven by fractional Brownian motion.  Collect Math. 2002;53, p. 55-–81.

\bibitem{Picchini et al. 2010} Picchini U, De Gaetano A, Ditlevsen S. Stochastic differential mixed-effects models. Scand. J. Statist.  2010;\textbf{37}, p. 67-–90.

\bibitem{Picchini and Ditlevsen 2011} Picchini U, Ditlevsen S. Practicle estimation of a high dimensional stochastic differential mixed-effects models. Comput. Statist. Data. Anal. 2011;\textbf{55}, p. 1426-–1444.

\bibitem{Prakasa 2003} Prakasa Rao BLS. Parametric estimation for linear stochastic differential equations driven by fractional Brownian motion. Random Oper. Stoch. Equ. 2003;11(3), p. 229-–242.

\bibitem{Prakasa 2010} Prakasa  Rao BLS. Statistical Inference for Fractional Diffusion Processes. Wiley Series in Probability and Statistics. Chichester: Wiley; 2010.

\bibitem{Schervish 1995} Schervish MJ.  Theory of Statistics. New York: Springer-Verlag; 1995.

\bibitem{Tsybakov 2009} Tsybakov AB.  Introduction to Nonparametric Estimation. New York: Springer; 2009.

\bibitem{Tudor and Viens 2007} Tudor CA, Viens FG. Statistical aspects of the fractional stochastic calculus. Ann. Statist. 2007; 35(3), p. 1183–-1212.

\bibitem{van der Vaart 1998} van der Vaart AW. Asymptotic statistics, vol. 3 of Cambridge Series in Statistical and Probabilistic Mathematics. Cambridge: Cambridge University Press; 1998.

\bibitem{Willinger et al. 1995} Willinger W, Taqqu MS, Leland WE, Wilson DV. : Self-similarity in high speed packet traffic: analysis and modelisation of ethernet traffic measurements Statistical Science. 1995; 10 , p.  67--85.

\end{thebibliography}

%\appendices
 \section*{Appendix A.}
  \begin{lemma}\label{LEM1} For all $c>0$ and $\alpha\in (0,1)$, we have
  \begin{equation*}
  \displaystyle \mathbb{P}\left(\abs{Z_1+Z_2}>c \right)\leq \mathbb{P}\left( \abs{Z_1}>(1-\alpha)c\right)+\mathbb{P}\left( \abs{Z_2}>\alpha c\right),
  \end{equation*}
  where $Z_1$ and $Z_2$ are two random variables.
   \end{lemma} \begin{lemma}\label{LEM2} Let $\displaystyle\left( X_n,~ n\geq 0\right)$ be a random sequence that converges weakly to a random variable $X$. Let $A$ be a Borel set such that $\displaystyle \mathbb{P}(X\in A)>0$ and $\displaystyle \mathbb{P}(X\in \delta A)=0$, where $\delta A$ denotes the boundary of the set $A$. For sufficiently large $n$, we have \begin{equation*} \displaystyle \mathbb{E}\abs{\chi_{(X_n\in A)}-\chi_{(X\in A)}} \leq \sqrt{2} \mathbb{P}(X\in A)^{1/2}\displaystyle\left\lbrace\mathbb{P}(X_n\notin A)^{1/2}+\mathbb{P}(X\notin A)^{1/2}\right\rbrace. \end{equation*} \end{lemma}
  \begin{proof}\label{proof.LEM2} Simple computations yield \begin{eqnarray}\label{Q4} \nonumber\displaystyle \mathbb{E}\abs{\chi_{(X_n\in A)}-\chi_{(X\in A)}} &=& \mathbb{E}\left(\chi_{(X_n\in A)}-\chi_{(X\in A)}\right)^2 \\ \nonumber\displaystyle &=& \mathbb{P}\left(X\in A\right)-\mathbb{P}\left(X,X_n\in A\right)+\mathbb{P}\left(X_n\in A\right)-\mathbb{P}\left(X,X_n\in A\right)\\ \nonumber\displaystyle &= & \mathbb{E}\left\lbrace \chi_{(X\in A)}\left(1-\chi_{(X_n\in A)}\right)\right\rbrace+\mathbb{E}\left\lbrace \chi_{(X_n\in A)}\left(1-\chi_{(X\in A)}\right)\right\rbrace\\ \displaystyle &\leq & \left[\mathbb{P}\left(X\in A\right)\mathbb{P}\left(X_n\notin A\right)\right]^{1/2}+\left[\mathbb{P}\left(X_n\in A\right)\mathbb{P}\left(X\notin A\right)\right]^{1/2}\!\!\!\!. \end{eqnarray} (\ref{Q4}) is justified by the Cauchy-Schwarz inequality. Since $\displaystyle\left\lbrace X_n\right\rbrace_{ n\geq 0}$ converges weakly to $X$, then by the Portmanteau lemma (see e.g. \cite{van der Vaart 1998}) we have
  \begin{equation}\label{T1} \displaystyle \mathbb{P}\left(X_n\in A\right)\leq 2\mathbb{P}(X\in A), \end{equation}
  for all $n\geq n_0$, where $n_0$ is sufficiently large. The desired result follows from (\ref{Q4}) and (\ref{T1}). \end{proof} \begin{lemma}\label{LEM3} Let $\displaystyle X_i,~i=1,\cdots,N$, be a sequence of i.i.d random variables with common density $f$. Assume that $f$ is continuous with compact support $A\subset \mathds{R}$. Let $\displaystyle A_j(h)=\left[hj,h(j+1)\right)$, $j=1,\cdots,J$ denote all Borel sets for which $\displaystyle \lambda\left( A_j(h)\cap A\right)\neq 0$. We have \begin{equation*} \displaystyle \lim_{h\rightarrow 0}\frac{1}{N}\sum_{i=1}^{N}\sum_{j=1}^{J}\mathbb{P}\left( X_i\in A_j(h) \right)^{1/2} = 0. \end{equation*} \end{lemma} \begin{proof}\label{proof.LEM1} Actually, \begin{eqnarray*} \displaystyle \sup_{i,j}\mathbb{P}\left( X_i\in A_j(h) \right) &=& \sup_{j}\int_{hj}^{h(j+1)} f(t)dt\\ \displaystyle &\leq & \sup_{t}f(t)h\rightarrow 0~\mbox{as}~h\rightarrow 0. \end{eqnarray*}
   Let $\varepsilon>0$. There exists $h_0>0$ such that $\displaystyle \mathbb{P}\left( X_i\in A_j(h) \right)<\varepsilon^2/j^4$, for all $i$ and $h\in (0,h_0)$. Hence \begin{equation*} \frac{1}{N}\sum_{i=1}^{N}\sum_{j=1}^{J}\mathbb{P}\left( X_i\in A_j(h) \right)^{1/2} \leq \varepsilon \left(\sum_{j\geq 1}\frac{1}{j^2} \right)<\infty. \end{equation*} \end{proof}

\end{document}